\documentclass[10pt]{amsart}
\usepackage{enumerate}

\newtheorem{theorem}{Theorem}
\newtheorem{lemma}[theorem]{Lemma}
\newtheorem{corollary}[theorem]{Corollary}

\theoremstyle{definition}
\newtheorem{definition}{Definition}
\newtheorem{example}{Example}

\newtheorem{prop}[theorem]{Proposition}

\theoremstyle{remark}
\newtheorem{remark}{Remark}

\numberwithin{equation}{section}

\begin{document}

\title[\tiny{\textbf{\textsf{Minimal surface systems, maximal surface systems
and special Lagrangian equations}}}]{Minimal surface systems, maximal surface systems
and special Lagrangian equations}

\author{Hojoo Lee}
\address{Department of Geometry and Topology, University of Granada, Granada, Spain}
\email{ultrametric@gmail.com}
\thanks{This work was supported by the National Research Foundation of Korea Grant funded
 by the Korean Government (Ministry of Education, Science and Technology) [NRF-2011-357-C00007] and
 in part by 2010 Korea-France STAR Program.}

 % ; Secondary 35J60, 53C50.
% \date{December 27, 2010 and, in revised form, .}


\begin{abstract}
  We extend Calabi's correspondence between minimal graphs in  Euclidean space ${\mathbb{R}}^{3}$ and maximal graphs in
 Lorentz-Minkowski space ${\mathbb{L}}^{3}$. We establish the  twin correspondence  between
 $2$-dimensional minimal graphs in Euclidean space ${\mathbb{R}}^{n+2}$ carrying a positive area-angle function and $2$-dimensional maximal graphs  in pseudo-Euclidean space  ${\mathbb{R}}^{n+2}_{n}$ carrying the same positive area-angle function.

 We generalize Osserman's Lemma on degenerate Gauss maps  of entire $2$-dimensional minimal graphs
 in ${\mathbb{R}}^{n+2}$ and offer several Bernstein-Calabi type theorems. A simultaneous application of  Harvey-Lawson Theorem
 on special Lagrangian equation and our extended Osserman's Lemma yields a geometric proof of J\"{o}rgens' Theorem on
 $2$-variables unimodular Hessian equation.

 We introduce the correspondence from $2$-dimensional minimal graphs in ${\mathbb{R}}^{n+2}$ to special Lagrangian
 graphs in ${\mathbb{C}}^{2}$, which induces an explicit correspondence from $2$-variables symplectic Monge-Amp\'{e}re 
 equations to $2$-variables unimodular Hessian equation.

\end{abstract}

\maketitle

 \tableofcontents

\section{Motivation and main results}

   The point of this paper is to investigate the minimal surface system in Euclidean space ${\mathbb{R}}^{n+2}$
   endowed with the metric ${dx_{1}}^{2} +\cdots + {dx_{n+2}}^{2}$ and the maximal surface system
    in pseudo-Euclidean space  ${\mathbb{R}}^{n+2}_{n}$ equipped with
  the metric ${dx_{1}}^{2} + {dx_{2}}^{2} -  {dx_{3}}^{2} -\cdots - {dx_{n+2}}^{2}$.

        The research on the minimal surface system is initiated in \cite{LO77, Oss69, Oss70}.  Lawson  and  Osserman \cite{LO77} studied  non-existence, non-uniqueness and irregularity of solutions of the minimal surface system.
 There have been extensive work in extending  Bernstein's Theorem to higher codimension \cite{Fu98, HSHV09, JX99, Kaw84,
 Mes00, Ni02, Sim77, TW02, Wa03, YU02} and studying the minimal surface system  \cite{LW03, LW08, Wa04a, Wa04b}.

 About one hundred years ago, Bernstein proved a truly beautiful theorem that the only entire solutions of the minimal surface equation in ${\mathbb{R}}^{3}$
\[
 0= \left( 1 + {f_{y}}^{2}   \right) f_{xx} - 2   f_{x} f_{y}   f_{xy}  +    \left( 1 + {f_{x}}^{2}  \right) f_{yy}
 \]
  are affine functions.  As illustrated in \cite{NS05}, extending Bernstein's Theorem in ${\mathbb{R}}^{3}$ is one of the central themes in the modern theory of minimal submanifolds.

  Calabi introduced a new extension of Bernstein's Theorem in ${\mathbb{R}}^{3}$. In 1970, he showed that 
  the only entire maximal graphs in $2+1$ dimensional space-time  ${\mathbb{L}}^{3}$ are spacelike planes. 
  In 1976, Cheng and Yau generalized Calabi's Theorem in  Lorentz space ${\mathbb{L}}^{n}$ for all 
  dimensions $n \geq 3$.

 Furthermore, Calabi \cite{Cal70} introduced an interesting duality between the minimal surface equation in  ${\mathbb{R}}^{3}$ and the maximal surface equation in ${\mathbb{L}}^{3}$. Al\'{i}as and Palmer \cite{AP01} employed
  Calabi's correspondence to show that the non-existence of entire non-planar minimal graphs in ${\mathbb{R}}^{3}$
is equivalent to the non-existence of entire non-planar maximal graphs in ${\mathbb{L}}^{3}$. Recently, Ara\'{u}jo and Leite \cite{AL09} discovered interesting results on the duality which is equivalent to Calabi's correspondence.

 In this paper, we extend the geometric integrability in the Calabi correspondence to higher codimension $n \geq 2$ by
  constructing the twin correspondence between  $2$-dimensional minimal graphs having a positive area-angle function (introduced in Section \ref{ntas}) in Euclidean space ${\mathbb{R}}^{n+2}$ and  $2$-dimensional maximal graphs having the same
  positive area-angle function in pseudo-Euclidean space ${\mathbb{R}}^{n+2}_{n}$.

 More explicitly, for the codimension $n \geq 2$, by the twin correspondence, a minimal graph $x{\mathbf{e}}_{1}+y{\mathbf{e}}_{2}+f_{1}(x,y){\mathbf{e}}_{3} +\cdots+f_{n}(x,y){\mathbf{e}}_{n+2}$ in ${\mathbb{R}}^{n+2}$ over the simply connected domain $\Omega$ satisfying the positive area-angle condition $\sum_{1 \leq i<j \leq n}  {\frac{ \partial \left( {f}_{i}, {f}_{j} \right)}{\partial (x, y)}}^{2}  <1$ associates a maximal graph $x{\mathbf{e}}_{1}+y{\mathbf{e}}_{2}+g_{1}(x,y){\mathbf{e}}_{3} +\cdots+g_{n}(x,y){\mathbf{e}}_{n+2}$ in ${\mathbb{R}}^{n+2}_{n}$ over the  domain $\Omega$ obeying the positive area-angle condition $\sum_{1 \leq i<j \leq n}  {\frac{ \partial \left( {g}_{i}, {g}_{j} \right)}{\partial (x, y)}}^{2}  <1$. 
 
  Our twin correspondence shows that the minimal surface system in ${\mathbb{R}}^{n+2}$ becomes the integrability condition for the maximal surface system in  ${\mathbb{R}}^{n+2}_{n}$. The twin correspondence with $n=2$ induces an explicit duality between the  $2$-variables special Lagrangian equation and the $2$-variables split special Lagrangian equation. We also construct the correspondence from minimal graphs in ${\mathbb{R}}^{n+2}$ to special Lagrangian graphs in ${\mathbb{C}}^{2}$.

The Bernstein problem for surfaces in higher codimension is to find conditions under which
 extremal graphs of functions from ${\mathbb{R}}^{2}$ to ${\mathbb{R}}^{n}$ are affine functions.
 Unlike Bernstein's Theorem in ${\mathbb{R}}^{3}$, for the higher codimension $n \geq 2$, there exist plenty of entire $2$-dimensional minimal non-planar  graphs in ${\mathbb{R}}^{n+2}$.

   Recently, Hasanis, Savas-Halilaj and Vlachos  \cite{HSHV09} showed that Bernstein's Theorem for  $2$-dimensional entire minimal graphs  in ${\mathbb{R}}^{n+2}$ with a positive area-angle function holds. The twin correspondence induces a Calabi type theorem for entire $2$-dimensional  maximal graphs in ${\mathbb{R}}^{n+2}_{n}$ with a positive area-angle function.

 We briefly sketch the structure of the rest of this paper.  In Section 2, we quickly review several notations,
 the minimal surface system in Euclidean space ${\mathbb{R}}^{n+2}$, and the maximal surface system  in pseudo-Euclidean space ${\mathbb{R}}^{n+2}_{n}$.

 In Section  \ref{GaussMap}, we study the generalized Gauss map of minimal surfaces in ${\mathbb{R}}^{n+2}$
 to investigate their global properties. Our extended Osserman's Lemma on degenerate Gauss maps of entire $2$-dimensional
 minimal graphs in ${\mathbb{R}}^{n+2}$ indicates that pairs of their height functions having
 nowhere zero Jacobian determinant on the whole plane contribute the degeneracy of their Gauss maps.

   Employing our extended Osserman's Lemma, we are able to prove several Bernstein type
  theorems for minimal surfaces with codimension $n \geq 2$.  If the height functions of
  entire $2$-dimensional minimal graphs in ${\mathbb{R}}^{n+2}$ are strictly monotone, they are planes.
 We  present a minimal-surface proof of J\"{o}rgens' Theorem that the only entire  solutions
 of the unimodular Hessian equation are quadratic polynomial functions.

 Section \ref{SR2R4} is devoted to the construction of the correspondence from  minimal graphs in ${\mathbb{R}}^{n+2}$ to
special Lagrangian graphs in ${\mathbb{C}}^{2}$.

 In Section \ref{SecTwin}, we construct the twin correspondence between  $2$-dimensional minimal graphs
 in ${\mathbb{R}}^{n+2}$ with a positive area-angle function and  $2$-dimensional maximal graphs
 in ${\mathbb{R}}^{n+2}_{n}$ with the same positive area-angle function.
  This generalize classical Calabi's correspondence between minimal graphs in ${\mathbb{R}}^{3}$ and maximal graphs in
 ${\mathbb{L}}^{3}$. An application of the twin correspondence yields a higher codimension extension
 of Calabi's Theorem in ${\mathbb{L}}^{3}$.

 In Section \ref{SyMA}, we introduce the special Lagrangian equation and the split special Lagrangian
 equation and see why they are geometrically equivalent to each other. Exploiting J\"{o}rgens' Theorem on
  the unimodular Hessian equation, we can classify all entire solutions of these
 two symplectic Monge-Amp\'{e}re equations. In particular, our Calabi type theorem states that any entire
 maximal gradient graph $(x,y,f_{x},f_{y})$ in ${\mathbb{R}}^{4}_{2}$ for some potential
 function $f:{\mathbb{R}}^{2} \rightarrow {\mathbb{R}}$ should be a spacelike plane.

 Finally, in Section \ref{SecHOL}, we prove the existence of simultaneous conformal coordinate transformations for
  twin graphs in ${\mathbb{R}}^{n+2}$ and ${\mathbb{R}}^{n+2}_{n}$.  The twin correspondence admits an explicit description via the induced holomorphic curves in their Weierstrass representation formulas.
  
  \textbf{Acknowledgement.} I would like to deeply thank my advisor Jaigyoung Choe for introducing me to the theory
 of minimal submanifolds, one of the most beautiful as well as useful branches of Mathematics. This paper was part 
 of my Ph.D. thesis. I also would like to express my gratitude to the referee for helpful comments and 
 valuable suggestions.

 \section{Preliminaries}

\subsection{Notations and assumptions}  \label{ntas}
 Throughout this paper,   $\Omega \subset {\mathbb{R}}^{2}$ denotes a simply connected domain.
All real-valued functions here are at least class of ${\mathcal{C}}^{2}$.
 Our ambient spaces are ${\mathbb{R}}^{n+2}$ and ${\mathbb{R}}^{n+2}_{n}$.  Here, ${\mathbb{R}}^{n+2}$ denotes the  Euclidean space endowed with the metric ${dx_{1}}^{2} +\cdots + {dx_{n+2}}^{2}$ and  ${\mathbb{R}}^{n+2}_{n}$ the  pseudo-Euclidean space  equipped with
  the metric ${dx_{1}}^{2} + {dx_{2}}^{2} -  {dx_{3}}^{2} -\cdots - {dx_{n+2}}^{2}$.

\begin{definition}  Let $n \geq 2$. Given a function $\phi=\left({\phi}_{1}, \cdots, {\phi}_{n}\right): \Omega \subset {\mathbb{R}}^{2} \rightarrow {\mathbb{R}}^{n}$, we use the notation
\[
 \Vert \mathcal{J} \Vert :=  \sqrt{ \sum_{1 \leq i<j \leq n}  {{\mathcal{J}}_{i,j}}^{2}}, \quad  {\mathcal{J}}_{i,j}:= \frac{ \partial \left( {\phi}_{i}, {\phi}_{j} \right)}{\partial (x, y)}.
\]
When the estimation $\Vert \mathcal{J} \Vert<1$ holds on the domain $\Omega$, we say that
 $\phi=\left({\phi}_{1}, \cdots, {\phi}_{n}\right)$ is a positive area-angle map
and that the graph of $\phi=\left({\phi}_{1}, \cdots, {\phi}_{n}\right)$ has a positive area-angle.  We call the number ${\Theta}_{\phi}=\cos^{-1}{\left(  \Vert \mathcal{J} \Vert \right)} \in \left(0, \frac{\pi}{2} \right]$ the area-angle of $\phi$.
\end{definition}

 \subsection{Minimal surface system in Euclidean space ${\mathbb{R}}^{n+2}$}

  Let   ${\mathbf{e}}_{1}=(1,0, \cdots, 0)$, $\cdots$,  ${\mathbf{e}}_{n+2}=(0,0, \cdots, 1)$ denote the
  standard basis in ${\mathbb{R}}^{n+2}$.

 \begin{prop}[\textbf{Osserman}, \cite{Oss69, Oss86}] \label{MSS1}
 The graph $\mathbf{\Phi}(x,y)=x {\mathbf{e}}_{1}+y{\mathbf{e}}_{2}+ f_{1}(x,y) {\mathbf{e}}_{3} +  \cdots + f_{n}(x,y)  {\mathbf{e}}_{n+2}$   of the height function $f=\left({f}_{1}, \cdots, {f}_{n}\right): \Omega \rightarrow {\mathbb{R}}^{n}$ becomes a minimal surface in ${\mathbb{R}}^{n+2}$ if and only if
 the minimal surface system holds:
\[
  G \frac{\partial^{2} f_{k}}{\partial x^{2}} - 2  F  \frac{\partial^{2} f_{k}}{\partial x \partial y} +  E \frac{\partial^{2} f_{k}}{\partial y^{2}}=0, \quad k \in \{1, \cdots, n\}.
\]
 Every ${\mathcal{C}}^{2}$ solution of the minimal
surface system is real analytic. Here, we set
\[
\begin{cases}
E=1 +  \left(
\frac{\partial f_{1}}{\partial x} \right)^{2} +\cdots + \left(
\frac{\partial f_{n}}{\partial x} \right)^{2}, \\
F= \quad \quad \frac{\partial f_{1}}{\partial x}  \frac{\partial f_{1}}{\partial y}
+ \cdots + \frac{\partial f_{n}}{\partial x}  \frac{\partial
f_{n}}{\partial y}, \\
G= 1 +   \left( \frac{\partial
f_{1}}{\partial y} \right)^{2} + \cdots + \left( \frac{\partial
f_{n}}{\partial y} \right)^{2}.
\end{cases}
\]
The patch $\mathbf{\Phi}$ induces the metric $ds_{\mathbf{\Phi}}^{2}=E dx^{2}+2F dx dy +Gdy^{2}$.
\end{prop}

\begin{definition} Let $\omega=\sqrt{EG-F^{2}}>0$ denote the area element. As in the codimension one case, we call
  $\frac{1}{\omega}$ the angle function of the minimal graph $\mathbf{\Phi}$ in ${\mathbb{R}}^{n+2}$.
\end{definition}

 \begin{prop} \label{MSS2} Let  $\left({\alpha}_{k}, {\beta}_{k} \right)=\left(\frac{\partial f_{k}}{\partial x}, \frac{\partial f_{k}}{\partial y} \right)$, $k \in \{1, \cdots, n\}$. \\
 (a) We can re-write the minimal surface system in the divergence-zero form:
\[
  \frac{\partial }{\partial x} \left( \frac{G}{\omega} {\alpha}_{k} - \frac{F}{\omega}  {\beta}_{k}  \right)+
   \frac{\partial }{\partial y} \left( \frac{E}{\omega} {\beta}_{k} - \frac{F}{\omega}  {\alpha}_{k}   \right)=0, \quad k \in \{1, \cdots, n\}.
\]
 (b) We also have $\frac{\partial }{\partial x} \left( \frac{G}{\omega} \right) = \frac{\partial }{\partial y} \left( \frac{F}{\omega} \right)$ and
 $\frac{\partial }{\partial x} \left(  \frac{F}{\omega}   \right) =  \frac{\partial }{\partial y} \left( \frac{E}{\omega}   \right)$.
\end{prop}

 \subsection{Maximal surface system in pseudo-Euclidean space ${\mathbb{R}}^{n+2}_{n}$}

 We consider ${\mathbf{e}}_{1}=(1,0, \cdots, 0)$, $\cdots$,  ${\mathbf{e}}_{n+2}=(0,0, \cdots, 1)$  as the
  standard basis in ${\mathbb{R}}^{n+2}_{n}$.

 \begin{prop} \label{MSS3}
  The spacelike graph $\mathbf{\widehat{\Phi}}(x,y)=x {\mathbf{e}}_{1}+y{\mathbf{e}}_{2}+ g_{1}(x,y) {\mathbf{e}}_{3} +  \cdots + g_{n}(x,y)  {\mathbf{e}}_{n+2}$   of the height function $g=\left({g}_{1}, \cdots, {g}_{n}\right): \Omega \rightarrow {\mathbb{R}}^{n}$ becomes a  maximal surface in  ${\mathbb{R}}^{n+2}_{n}$  if and only if $g$ satisfies the maximal surface
  system:
\[
  \widehat{G} \frac{\partial^{2} g_{k}}{\partial x^{2}} - 2    \widehat{F}  \frac{\partial^{2} g_{k}}{\partial x \partial y} +
     \widehat{E} \frac{\partial^{2} g_{k}}{\partial y^{2}}=0, \quad k \in \{1, \cdots, n\},
\]
or equivalently,
\[
  \frac{\partial }{\partial x} \left( \frac{\widehat{G}}{\widehat{\omega}} {\widehat{\alpha}}_{k} - \frac{\widehat{F}}{\widehat{\omega}}  {\widehat{\beta}}_{k}  \right)+
   \frac{\partial }{\partial y} \left( \frac{\widehat{E}}{\widehat{\omega}} {\widehat{\beta}}_{k} - \frac{\widehat{F}}{\widehat{\omega}}  {\widehat{\alpha}}_{k}   \right)=0, \quad  \left({\widehat{\alpha}}_{k}, {\widehat{\beta}}_{k} \right)=\left(\frac{\partial g_{k}}{\partial x}, \frac{\partial g_{k}}{\partial y} \right).
\]
Here, we write
\[
\begin{cases}
\widehat{E}=1-  \left( \frac{\partial g_{1}}{\partial x} \right)^{2} -\cdots - \left( \frac{\partial g_{n}}{\partial x} \right)^{2}, \\
\widehat{F}= \;\;\;   - \frac{\partial g_{1}}{\partial x}  \frac{\partial g_{1}}{\partial y} - \cdots - \frac{\partial g_{n}}{\partial x}  \frac{\partial g_{n}}{\partial y}, \\
\widehat{G}= 1 -  \left( \frac{\partial g_{1}}{\partial y} \right)^{2} - \cdots - \left( \frac{\partial g_{n}}{\partial y} \right)^{2}.
\end{cases}
\]
 We assume the spacelike condition $\widehat{E}\widehat{G}-{\widehat{F}}^{2}>0$.  The patch $\mathbf{\widehat{\Phi}}$ induces the Riemannian metric $ds_{\mathbf{\widehat{\Phi}}}^{2}= \widehat{E}dx^{2}+2 \widehat{F}dx dy + \widehat{G}dy^{2}$. We set $\widehat{\omega}=\sqrt{\widehat{E}\widehat{G}-{\widehat{F}}^{2}}$. We call $\frac{1}{\widehat{\omega}}$ the angle function of the maximal graph $\mathbf{\widehat{\Phi}}$ in ${\mathbb{R}}^{n+2}_{n}$.
   Also, we obtain two identities $\frac{\partial }{\partial x} \left( \frac{\widehat{G}}{\widehat{\omega}} \right) = \frac{\partial }{\partial y} \left( \frac{\widehat{F}}{\widehat{\omega}} \right)$ and $\frac{\partial }{\partial x} \left(  \frac{\widehat{F}}{\widehat{\omega}}   \right) =  \frac{\partial }{\partial y} \left( \frac{\widehat{E}}{\widehat{\omega}}   \right)$.
\end{prop}
The proofs of the above results  are analogous to the minimal surface system.

\section{Two dimensional minimal graphs in ${\mathbb{R}}^{n+2}$}  \label{BL}

\subsection{Extension of Osserman's Lemma}  \label{GaussMap}

 Bernstein's Theorem in ${\mathbb{R}}^{3}$ is essentially concerned with the Gauss map. We first introduce the
 generalized Gauss map of minimal surfaces in ${\mathbb{R}}^{n+2}$ \cite{CO67, HO80, HO83, Oss69, Oss86}. Inside the $n+1$ dimensional complex projective space  ${\mathbb{C}}{\mathbb{P}}^{n+1}$, we take the complex hyperquadric
\[
  {\mathcal{Q}}_{n}:=\{ z=[{z}_{1}: \cdots: {z}_{n+2} ] \in  {\mathbb{C}}{\mathbb{P}}^{n+1}  : {{z}_{1}}^{2} + \cdots +  {{z}_{n+2}}^{2} =0 \}.
\]

\begin{definition}[\textbf{Generalized Gauss map of minimal surfaces in ${\mathbb{R}}^{n+2}$}]
Let $\Sigma$ be a minimal surface in ${\mathbb{R}}^{n+2}$. Consider a conformal harmonic immersion $X:\Sigma \rightarrow {\mathbb{R}}^{n+2}$, $\xi \mapsto X(\xi)$. The Gauss map of $\Sigma$ is the map
$\mathcal{G}:\Sigma \rightarrow {\mathcal{Q}}_{n} \subset {\mathbb{C}}{\mathbb{P}}^{n+1}$ defined by
\[
   \mathcal{G}(\xi)=\left[ \;   \overline{ \frac{\partial X}{\partial {\xi}} } \; \right]
   = \left[ \frac{\partial X}{\partial {\xi}_{1}}  + i \frac{\partial X}{\partial {\xi}_{2}} \right] \in  {\mathbb{C}}{\mathbb{P}}^{n+1}.
\]
The map $\mathcal{G}$ is independent of the choice of the conformal parameter $\xi$.
\end{definition}

\begin{lemma}[\textbf{Osserman's Lemma}, \cite{Oss70, Oss86}] \label{oentire}
Let $\Sigma$ be an entire minimal graph  of the function $f=\left(f_{1}, \cdots, f_{n} \right):{\mathbb{R}}^{2} \rightarrow {\mathbb{R}}^{n}$, $n \geq 1$
\[
 x{\mathbf{e}}_{1}+y{\mathbf{e}}_{2}+f_{1}(x,y){\mathbf{e}}_{3} +\cdots+f_{n}(x,y){\mathbf{e}}_{n+2}.
\]
(a) The Gauss map image $\mathcal{G}\left(\Sigma\right)$ lies on a hyperplane $z_{2}
= {\lambda} z_{1}$ for some  ${\lambda}  \in {\mathbb{C}}-\mathbb{R}$. \\
(b) There exists a non-singular linear transformation $\left(u,v\right)\mapsto \left(x,y\right)=\left(u,au+bv\right)$  for some
$a, b \in \mathbb{R}$, $b \neq 0$ such that $u+iv$ is a global isothermal parameter for $\Sigma$.
\end{lemma}

 Osserman's Lemma on Gauss maps is highly useful. It is  used in \cite{Fu98, HSHV09, Kaw84, Ni02} for the
  proofs of various Bernstein type theorems for entire minimal graphs in  ${\mathbb{R}}^{4}$.
 In Lemma \ref{entire}, we present an extension of Osserman's Lemma. It indicates that, given an entire two dimensional minimal graph in ${\mathbb{R}}^{n+2}$, pairs of its height functions having non-zero Jacobian determinant on the whole plane also contributes the degeneracy of its Gauss map. The following simple observation on the sign of the Jacobian determinants of entire holomorphic graphs is the motivation of Lemma \ref{entire}:

\begin{remark} \label{ehg} We consider an entire holomorphic graph $\Sigma$ in ${\mathbb{R}}^{2k+2}$ given by
\begin{eqnarray*}
 && F_{1}(x,y){\mathbf{e}}_{1}+F_{2}(x,y){\mathbf{e}}_{2}+F_{3}(x,y){\mathbf{e}}_{3} +\cdots+F_{n+2}(x,y){\mathbf{e}}_{n+2} \\
 &=& x{\mathbf{e}}_{1}+ y {\mathbf{e}}_{2}+ {\mathrm{Re}(\phi_{1})}{\mathbf{e}}_{3} +
 {\mathrm{Im}(\phi_{1})}{\mathbf{e}}_{4} + \cdots + {\mathrm{Re}(\phi_{k})}{\mathbf{e}}_{2k+1} +
 {\mathrm{Im}(\phi_{k})}{\mathbf{e}}_{2k+2},
\end{eqnarray*}
where $\phi_{1}(z), \cdots, \phi_{k}(z): {\mathbb{C}} \to {\mathbb{C}}$, $z=x+iy \in \mathbb{R}+i \mathbb{R}$ are entire holomorphic functions.  The entire graph $\Sigma$ is minimal in ${\mathbb{R}}^{2k+2}$. We observe that the Jacobian
${\mathcal{J}}_{1,2}:= \frac{ \partial \left( {F}_{1}, {F}_{2} \right)}{\partial (x, y)}
 = \frac{ \partial \left( x,y \right)}{\partial (x, y)}=1$ is positive  on the whole plane ${\mathbb{R}}^{2}$ and that,  for each $i \in \{1, \cdots, k\}$, the Jacobian
${\mathcal{J}}_{2i+1,2i+2}:= \frac{ \partial \left( {F}_{2i+1}, {F}_{2i+2} \right)}{\partial (x, y)}$
is always non-negative on the whole plane ${\mathbb{R}}^{2}$. It is because, for any  differentiable function
 $\phi: {\mathbb{C}} \to {\mathbb{C}}$ given by
$\xi={\xi}_{1}+i{\xi}_{2} \in \mathbb{R}+i \mathbb{R} \mapsto {\phi}_{1}({\xi}_{1}, {\xi}_{2})+i{\phi}_{2}({\xi}_{1}, {\xi}_{2}) \in \mathbb{R}+i \mathbb{R}$, we have the identity
 \[
 \frac{ \partial \left( {\phi}_{1}, {\phi}_{2} \right)}{\partial ({\xi}_{1}, {\xi}_{2})}=
 { \left\vert \frac{\partial \phi}{\partial \xi} \right\vert }^{2} -
 { \left \vert \frac{\partial \phi}{\partial \overline{\xi}} \right\vert }^{2}.
 \]
If $\phi$ is holomorphic, then $\frac{ \partial \left({\phi}_{1},{\phi}_{2}\right)}{\partial ({\xi}_{1},{\xi}_{2})} \geq 0$.
If $\phi$ is anti-holomorphic, then $\frac{\partial \left({\phi}_{1},{\phi}_{2}\right)}{\partial({\xi}_{1},{\xi}_{2})}\leq 0$.\\
\end{remark}

\begin{lemma}[\textbf{Extended Osserman's Lemma}] \label{entire}
Let $\Sigma$ be an entire minimal graph  of  $f=\left(f_{1}, \cdots, f_{n} \right):{\mathbb{R}}^{2} \rightarrow {\mathbb{R}}^{n}$, $n \geq 1$
\begin{eqnarray*}
 \mathbf{\Phi}(x,y)
 &=&  F_{1}(x,y){\mathbf{e}}_{1}+F_{2}(x,y){\mathbf{e}}_{2}+F_{3}(x,y){\mathbf{e}}_{3} +\cdots+F_{n+2}(x,y){\mathbf{e}}_{n+2} \\
 &=&  x{\mathbf{e}}_{1}+y{\mathbf{e}}_{2}+f_{1}(x,y){\mathbf{e}}_{3} +\cdots+f_{n}(x,y){\mathbf{e}}_{n+2}.
\end{eqnarray*}
 Whenever there exists a pair $(i, j)$ with  $1 \leq i < j \leq n+2$ such that the Jacobian
 \[
 {\mathcal{J}}_{i,j}:= \frac{ \partial \left( {F}_{i}, {F}_{j} \right)}{\partial (x, y)}
\]
 is positive or negative on the entire plane ${\mathbb{R}}^{2}$, the following two statements hold:\\
(a) The image $\mathcal{G}\left(\Sigma\right)$  of  $\Sigma$ under the Gauss map lies on a hyperplane of the form
\[
    z_{i} = {\lambda}_{(i,j)} z_{j}
\]
for some constant ${\lambda}_{(i,j)}  \in {\mathbb{C}}-\mathbb{R}$. \\
(b) For some constant ${\lambda}  \in {\mathbb{C}}-\mathbb{R}$, we have an analogue of Cauchy-Riemann equations
\[
    \frac{\partial F_{i}}{\partial x} + \lambda  \frac{\partial F_{i}}{\partial y}
    = {\lambda}_{(i,j)} \left(   \frac{\partial F_{j}}{\partial x} + \lambda  \frac{\partial F_{j}}{\partial y} \right).
\]
\end{lemma}

\begin{proof} \textbf{Step A.} Following the arguments in \cite{Oss86}, we prepare the global conformal coordinate chart for the entire graph $\Sigma$.  We begin with the two identities:
 \[
 \frac{\partial }{\partial x} \left(  \frac{F}{\omega}   \right) =  \frac{\partial }{\partial y} \left( \frac{E}{\omega}   \right) \quad \text{and} \quad
 \frac{\partial }{\partial x} \left( \frac{G}{\omega} \right) = \frac{\partial }{\partial y} \left( \frac{F}{\omega} \right).
 \]
 Since ${\mathbb{R}}^{2}$ is simply connected, Poincar\'{e} Lemma guarantees the existence of functions
 $M, N:{\mathbb{R}}^{2} \rightarrow \mathbb{R}$ satisfying the equalities
 \[
 M_{x} = \frac{E}{\omega}, M_{y} = \frac{F}{\omega}, N_{x} = \frac{F}{\omega}, N_{y} = \frac{G}{\omega}.
 \]
 We then introduce the coordinate transformation
 \[
  \Psi: (x,y) \mapsto \left( {\xi}_{1}, {\xi}_{2} \right):=\left(x+M(x,y),
 y+N(x,y) \right).
 \]
 One computes the Jacobian determinant of the transformation  $\Psi$:
 \[
   J_{\Psi}= \frac{ \partial \left(  {\xi}_{1}, {\xi}_{2} \right)}{\partial (x, y)} = \det
   \begin{pmatrix}    1+ \frac{E}{\omega}   &   \frac{F}{\omega}  \\     \frac{F}{\omega}  &  1+  \frac{G}{\omega}
   \end{pmatrix}    = 2+ \frac{E+G}{\omega}>2.
 \]
 Since $ J_{\Psi}=\frac{ \partial \left(  {\xi}_{1}, {\xi}_{2} \right)}{\partial (x, y)}>0$, we obtain the existence of
 the local inverse $\left( {\xi}_{1}, {\xi}_{2} \right) \mapsto (x,y)$.
  Furthermore, Osserman \cite{Oss86} proved that the mapping
  \[
  \Psi: (x,y) \mapsto \left( {\xi}_{1}, {\xi}_{2} \right)
  \]
  becomes a diffeomorphism from ${\mathbb{R}}^{2}$ to itself. A straightforward computation shows that
  $\xi=\xi_{1}+i\xi_{2}$ becomes a global isothermal parameter on $\Sigma$:
 \[
 ds_{\Sigma}^{2}=\frac{\omega}{J_{\Psi}} \left( d{{\xi}_{1}}^{2} +  d{{\xi}_{2}}^{2} \right).
\]
Now, the  conformal immersion $X:=\mathbf{\Phi} \circ {\Psi}^{-1}: {\mathbb{R}}^{2} \rightarrow {\mathbb{R}}^{n+2}$:
\[
X \left(  {\xi}_{1}, {\xi}_{2} \right) = \sum_{k=1}^{n+2} F_{k}(x \left(  {\xi}_{1}, {\xi}_{2} \right),y \left(  {\xi}_{1}, {\xi}_{2} \right)) {\mathbf{e}}_{k}
\]
induces the entire holomorphic functions
\[
  {\phi}_{k}(\xi)= 2 \frac{\partial F_{k}}{\partial \xi}, \quad k \in \{1, \cdots, n\}.
\]
 \textbf{Step B.} (Proof of (a)) According to the assumption  that the Jacobian determinant 
\[ 
 {\mathcal{J}}_{i,j}:= \frac{ \partial \left( {F}_{i}, {F}_{j} \right)}{\partial (x, y)}
 \]
  is positive or negative on the entire plane ${\mathbb{R}}^{2}$, we find that
\[
\mathrm{Im}\left( {\phi}_{i} \overline{{\phi}_{j}} \, \right) =
\frac{\partial (F_{i}, F_{i})}{\partial \left( {\xi}_{1}, {\xi}_{2} \right)}
=  \frac{\partial (F_{i}, F_{j})}{\partial (x,y)} \frac{\partial (x,y)}{\partial \left(  {\xi}_{1}, {\xi}_{2} \right)} = \frac{{\mathcal{J}}_{i,j}}{J_{\Psi}}
\]
is positive or negative on $\mathbb{C}$. In particular, two holomorphic functions ${\phi}_{i}$ and $ {\phi}_{j}$ have no zeros.
So, the holomorphic function $\frac{ {\phi}_{i} }{  {\phi}_{j} }$ is entire. Furthermore, the function
\[
\mathrm{Im}\left( \frac{{\phi}_{i}} { {\phi}_{j}} \, \right) =
   \frac{1}{ {\vert {\phi}_{j} \vert}^{2} } \mathrm{Im}\left( {\phi}_{i} \overline{ {\phi}_{j}} \, \right)
 \]
has the same (nowhere-zero) sign on the entire plane $\mathbb{C}$. Since the holomorphic function $\frac{ {\phi}_{i} }{  {\phi}_{j} }$ is entire, this means that $\frac{ {\phi}_{i} }{  {\phi}_{j} }$ becomes a non-zero constant function. Write
 ${\phi}_{i}= c {\phi}_{j}$ for some $c \in {\mathbb{C}}^{*}=\mathbb{C}-\{0\}$. Since $\mathrm{Im} \; c=\mathrm{Im}\left( \frac{{\phi}_{i}} { {\phi}_{j}} \, \right) \neq 0$, we know $c \in {\mathbb{C}}-\mathbb{R}$.
 By the definition of the Gauss map $\mathcal{G}$:
 \[
  [{z}_{1}: \cdots: {z}_{n+2} ] = \mathcal{G}(\xi)=\left[ \;   \overline{ \frac{\partial X}{\partial {\xi}} } \; \right]
= \left[\, \overline{{\phi}_{1}}, \cdots, \overline{{\phi}_{n+2}} \, \right],
\]
 we see that the image $\mathcal{G}\left(\Sigma\right)$ lies on the hyperplane $z_{i} = \lambda z_{j}$ with $\lambda=\overline{c}\in {\mathbb{C}}-\mathbb{R}$. \\
 \textbf{Step C.} (Proof of (b)) We recall that the coordinate change $\Psi: (x,y) \mapsto \left( {\xi}_{1}, {\xi}_{2} \right)$
 gives the conformal immersion $X=\mathbf{\Phi} \circ {\Psi}^{-1}$. The Chain Rule yields
\[
   \frac{G}{\omega} \frac{\partial {\mathbf{\Phi}} }{\partial x}  +
  \left(  i - \frac{F}{\omega} \right)   \frac{\partial {\mathbf{\Phi}} }{\partial y}
   = \left(  1+ \frac{G}{\omega} + i \frac{F}{\omega} \right)
    \left( \frac{\partial {  {X}} }{\partial {{\xi}_{1}}} + i \frac{\partial { {X}} }{\partial {{\xi}_{2}}} \right),
  \]
 which means that the graph $\mathbf{\Phi}(x,y)$ has the Gauss map
\begin{eqnarray*}
  \mathcal{G}(x,y)&=& \left[ \;   \overline{ \frac{\partial X}{\partial {\xi}} } \; \right] \\
 &=& \left[\frac{G}{\omega} \frac{\partial {\mathbf{\Phi}} }{\partial x} +
  \left(  i - \frac{F}{\omega} \right)   \frac{\partial {\mathbf{\Phi}} }{\partial y} \right] \\
   &=& \left[   \frac{G}{\omega} \frac{\partial F_{1}}{\partial x} +
  \left(  i - \frac{F}{\omega} \right)   \frac{\partial F_{1}}{\partial y}, \cdots,
    \frac{G}{\omega} \frac{\partial F_{n+2}}{\partial x} +
  \left(  i - \frac{F}{\omega} \right)   \frac{\partial F_{n+2}}{\partial y}  \right] \\
   &=& \left[  \frac{G}{\omega},  i- \frac{F}{\omega} ,  \frac{G}{\omega} \frac{\partial f_{1}}{\partial x} +
  \left(  i - \frac{F}{\omega} \right)   \frac{\partial f_{1}}{\partial y}, \cdots, \frac{G}{\omega} \frac{\partial f_{n}}{\partial x} + \left(  i - \frac{F}{\omega} \right)   \frac{\partial f_{n}}{\partial y} \right].
    \end{eqnarray*}
 Since $z_{2}= {\lambda} z_{1}$ on the Gauss map image $\mathcal{G}\left(\Sigma\right)$  for some constant ${\lambda}  \in {\mathbb{C}}-{\mathbb{R}}$, the equality $i- \frac{F}{\omega} = {\lambda} \frac{G}{\omega}$
 holds on $\mathbb{C}$. Since  $z_{i} = {\lambda}_{(i,j)} z_{j}$ on  $\mathcal{G}\left(\Sigma\right)$, we obtain
 \[
 \frac{G}{\omega} \frac{\partial f_{i}}{\partial x} + \left(  i - \frac{F}{\omega} \right)   \frac{\partial f_{i}}{\partial y} = {\lambda}_{(i,j)} \left( \frac{G}{\omega} \frac{\partial f_{j}}{\partial x} + \left(  i - \frac{F}{\omega} \right)   \frac{\partial f_{j}}{\partial y} \right).
 \]
Since $i- \frac{F}{\omega} = {\lambda} \frac{G}{\omega}$ and  $\frac{G}{\omega}>0$, it reduces to
$\frac{\partial F_{i}}{\partial x} + \lambda  \frac{\partial F_{i}}{\partial y}
    = {\lambda}_{(i,j)} \left(   \frac{\partial F_{j}}{\partial x} + \lambda  \frac{\partial F_{j}}{\partial y} \right)$.
\end{proof}

The above proof of Lemma \ref{entire} contains an explicit formula for the  Gauss map:

\begin{prop} \label{GM}
 The minimal graph $(x,y) \in \Omega \subset {\mathbb{R}}^{2} \mapsto \mathbf{\Phi}(x,y)= x{\mathbf{e}}_{1}+y{\mathbf{e}}_{2}+f_{1}(x,y){\mathbf{e}}_{3} +\cdots+f_{n}(x,y){\mathbf{e}}_{n+2}$  in ${\mathbb{R}}^{n+2}$ has the Gauss map $\mathcal{G}:\Omega \rightarrow {\mathcal{Q}}_{n} \subset {\mathbb{C}}{\mathbb{P}}^{n+1}$:
\begin{eqnarray*}
  \mathcal{G}(x,y)  &=& \left[  \frac{G}{\omega}, \;  i- \frac{F}{\omega} ,  \frac{G}{\omega} \frac{\partial f_{1}}{\partial x} +  \left(  i - \frac{F}{\omega} \right)   \frac{\partial f_{1}}{\partial y}, \cdots, \frac{G}{\omega} \frac{\partial f_{n}}{\partial x} + \left(  i - \frac{F}{\omega} \right)   \frac{\partial f_{n}}{\partial y} \right]  \\
 &=& \left[  1-i\frac{F}{\omega}, i\frac{E}{\omega},  \left(  1 - i \frac{F}{\omega} \right) \frac{\partial f_{1}}{\partial x} + i \frac{E}{\omega}  \frac{\partial f_{1}}{\partial y}, \cdots,  \left(  1 - i \frac{F}{\omega} \right) \frac{\partial f_{n}}{\partial x} + i \frac{E}{\omega} \frac{\partial f_{n}}{\partial y}  \right].
    \end{eqnarray*}
Here, $ds_{\mathbf{\Phi}}^{2}=E dx^{2}+2F dx dy + Gdy^{2}$ is the induced metric of the graph $\mathbf{\Phi}$.
\end{prop}
\begin{proof}
In the notations in the proof of Lemma \ref{entire}, the second identity follows from
\[
 \left(  1 - i \frac{F}{\omega} \right)   \frac{\partial {\mathbf{\Phi}} }{\partial x}  +  i \frac{E}{\omega}
  \frac{\partial {\mathbf{\Phi}} }{\partial y}  = \left(  1+ \frac{E}{\omega} - i \frac{F}{\omega} \right)
    \left( \frac{\partial {  {X}} }{\partial {{\xi}_{1}}} + i \frac{\partial { {X}} }{\partial {{\xi}_{2}}} \right).
\]
\end{proof}

Using Lemma \ref{entire} and Proposition \ref{GM}, we present a minimal-surface proof of

\begin{corollary}[\textbf{J\"{o}rgens Theorem}, \cite{Jo54}] \label{Jo}
 The only entire  $\mathcal{C}^{3}$ functions satisfying the unimodular Hessian equation
$F_{xx}F_{yy}-{F_{xy}}^{2}=1$ are quadratic polynomial functions.
\end{corollary}

\begin{proof} \label{PFjo}  According to the Harvey-Lawson Theorem \cite{HL82}, the surface $\Sigma$
 \[
{F}_{1}(x,y){\mathbf{e}}_{1}+{F}_{2}(x,y){\mathbf{e}}_{2}+{F}_{3}(x,y){\mathbf{e}}_{3} +{F}_{4}(x,y){\mathbf{e}}_{4} = x{\mathbf{e}}_{1}+y{\mathbf{e}}_{2}+F_{x}{\mathbf{e}}_{3} +F_{y}{\mathbf{e}}_{4}
 \]
 becomes an entire minimal graph in  ${\mathbb{R}}^{4}$. We can use Proposition \ref{GM} to show that its Gauss map $\mathcal{G}:\Sigma \rightarrow {\mathcal{Q}}_{2} \subset {\mathbb{C}}{\mathbb{P}}^{3}$ reads, for some constant
 $\epsilon \in \{-1, 1\}$,
\[
  \mathcal{G}(x,y) = \left[z_{1}, z_{2}, z_{3}, z_{4} \right]
 = \left[  \epsilon F_{yy}, \;  i- \epsilon F_{xy} ,  \epsilon + i F_{xy} ,   i  F_{yy} \right].
\]
The Gauss map image  $\mathcal{G}\left(\Sigma\right)$ lies on two hyperplanes $z_{2}= i \epsilon z_{3}$ and $z_{4}= i \epsilon z_{1}$. Since Jacobians  ${\mathcal{J}}_{1,2}:= \frac{ \partial \left( {{F}}_{1}, {{F}}_{2} \right)}{\partial (x, y)}
=1$ and ${\mathcal{J}}_{3,4}:= \frac{ \partial \left( {{F}}_{3}, {{F}}_{4} \right)}{\partial (x, y)} =F_{xx}F_{yy}-{F_{xy}}^{2}=1$ are positive on the whole plane ${\mathbb{R}}^{2}$, (a) in Lemma \ref{entire} implies that the Gauss map image
$\mathcal{G}\left(\Sigma\right)$ lies on two hyperplanes $z_{1} = {\lambda}_{(1,2)} z_{2}$ and $z_{3} = {\lambda}_{(3,4)} z_{4}$ for some constants ${\lambda}_{(1,2)}  \in {\mathbb{C}}-\mathbb{R}$ and ${\lambda}_{(3,4)}  \in {\mathbb{C}}-\mathbb{R}$. We conclude that its Gauss map $\mathcal{G}$ is constant.
\end{proof}

\begin{remark} We here verify the claim on the Gauss map appeared in the above proof:
\[
 \mathcal{G}(x,y) = \left[   \epsilon F_{yy}, \;  i- \epsilon F_{xy} ,  \epsilon + i F_{xy} ,   i  F_{yy}  \right].
\]
 We need to compute the pullback metric ${ds}^{2}$ of the graph $\Sigma$:
\begin{eqnarray*}
 {ds}^{2}&=&   E dx^{2}+2F dx dy + Gdy^{2} \\
 &=& \left(1 + {F_{xx}}^{2} +   {F_{xy}}^{2} \right) dx^{2}  + 2 F_{xy}\left( F_{xx}+F_{yy} \right)dx dy + \left( 1 + {F_{xy}}^{2} +   {F_{yy}}^{2} \right) dy^{2}.
 \end{eqnarray*}
 Write $\omega=\sqrt{EG-F^{2}}>0$. It follows from $F_{xx}F_{yy}-{F_{xy}}^{2}=1$ that
 \begin{eqnarray*}
 \omega &=& \sqrt{ \left( 1 + {F_{xx}}^{2} +   {F_{xy}}^{2} \right)\left( 1 + {F_{xy}}^{2} +   {F_{yy}}^{2} \right)
 - \left( F_{xx} F_{xy} + F_{xy } F_{yy}  \right)^{2}      }  \\
 &=& \sqrt{ \left( {F_{xx}}+{F_{yy}}\right)^{2}+ \left(1 - \left(  F_{xx}F_{yy}-{F_{xy}}^{2} \right) \right) ^{2}     } \\
 &=& \vert {F_{xx}}  +  {F_{yy}} \vert.
 \end{eqnarray*}
Since $F_{xx}F_{yy}-{F_{xy}}^{2}>0$, the continuous function $F_{xx}+F_{yy}$ vanishes nowhere.
We can find a constant  $\epsilon \in \{-1, 1\}$ satisfying   ${F_{xx}}  +  {F_{yy}} = \epsilon \omega$ on
${\mathbb{R}}^{2}$ and obtain
\[
  \frac{F}{\omega} = \frac{F_{xy}\left( F_{xx}+F_{yy} \right)}{\omega} = \epsilon F_{xy}
\]
and
\[
  \frac{G}{\omega} = \frac{1 + {F_{xy}}^{2} +   {F_{yy}}^{2}}{\omega} = \frac{F_{yy}\left( F_{xx}+F_{yy} \right)}{\omega} = \epsilon F_{yy}.
\]
Now, by Proposition \ref{GM}, the Gauss map of the entire minimal graph  $\Sigma$ in ${\mathbb{R}}^{4}$ reads
\begin{eqnarray*}
  \mathcal{G}(x,y) &=& \left[  \frac{G}{\omega}, \;  i- \frac{F}{\omega} ,  \frac{G}{\omega} F_{xx} + \left(  i - \frac{F}{\omega} \right)  F_{xy},  \frac{G}{\omega} F_{yx} + \left(  i - \frac{F}{\omega} \right)  F_{yy} \right]   \\
  &=&   \left[   \epsilon F_{yy}, \;  i- \epsilon F_{xy} ,  \epsilon + i F_{xy} ,   i  F_{yy}  \right].
 \end{eqnarray*}
\end{remark}

\begin{remark} The proof of Corollary \ref{Jo} used the Harvey-Lawson Theorem \cite{HL82} that the gradient graph of a  solution to unimodular Hessian equation in two variables is a special Lagrangian surface in ${\mathbb{C}}^{2}$, which
has the area-minimizing property. On the other hand, Warren \cite[Theorem 3.2]{War10} showed that
the gradient graph of a convex solution to unimodular Hessian equation becomes a spacelike volume-maximizing submanifold of the pseudo-Euclidean space endowed with a suitably chosen pseudo-metric. See also Mealy's earlier result \cite{Mea91},
\cite[Theorem 6.3]{HL10} and Section \ref{SyMA}.
\end{remark}

\begin{remark}
In Section \ref{SyMA}, we show that Bernstein type theorems for symplectic Monge-Amp\'{e}re equations in two variables can be reduced to J\"{o}rgens' Theorem. Other proofs of J\"{o}rgens' Theorem
 or Bernstein type theorem for special Lagrangian equations are given by Fu \cite{Fu98} and Yuan \cite{YU02}.
\end{remark}

 There are many applications of our extended Osserman's Lemma. One can show that if the height functions of entire $2$-dimensional minimal graphs in ${\mathbb{R}}^{n+2}$ are strictly monotone, they are planes:

\begin{theorem} \label{ent1}
 Let  $\Sigma$ be an entire minimal graph of $f=\left(f_{1}, \cdots, f_{n} \right):{\mathbb{R}}^{2} \rightarrow {\mathbb{R}}^{n}$.\\
 (a) Suppose that, for each $k \in \{1, \cdots, n\}$, one of the following inequalities holds\\
\[
  \frac{\partial f_{k}}{\partial x}>0,  \;\;  \frac{\partial f_{k}}{\partial x}<0, \;
  \frac{\partial f_{k}}{\partial y}>0,   \;\;  \frac{\partial f_{k}}{\partial y}<0
\]
  on the whole plane ${\mathbb{R}}^{2}$. Then, the entire graph $\Sigma$ in ${\mathbb{R}}^{n+2}$ is a plane. \\
 (b)  Suppose that, for some fixed $k \in \{1, \cdots, n\}$,  either $\frac{\partial f_{k}}{\partial x}$ or $\frac{\partial f_{k}}{\partial y}$ is a non-zero constant function on the whole plane ${\mathbb{R}}^{2}$. Then, the function $f_{k}$ is affine.
\end{theorem}

\begin{proof}
  (a) By  Lemma \ref{oentire} or Lemma \ref{entire}, there exists a constant ${\lambda}  \in {\mathbb{C}}^{*}$ such that the equality $z_{2}= {\lambda} z_{1}$  holds on the Gauss map image $\mathcal{G}\left(\Sigma\right)$. We fix $k \in \{1, \cdots, n\}$. \\
  (a1) When $\frac{\partial f_{k}}{\partial y}= \frac{ \partial \left( x, {f}_{k} \right)}{\partial (x, y)}$ never vanishes, take $(i,j)=(k+2,1)$ in (a) of Lemma \ref{entire}. Thus,  $\mathcal{G}\left(\Sigma\right)$ lies on a hyperplane $z_{k+2} = {\lambda}_{k+2} z_{1}$ for some constant ${\lambda}_{k+2}  \in {\mathbb{C}}^{*}$. \\
  (a2) In the case  when $\frac{\partial f_{k}}{\partial x}=\frac{ \partial \left(  {f}_{k},y \right)}{\partial (x, y)}$ never vanishes, take $(i,j)=(k+2,2)$ in (a) of Lemma \ref{entire}. So,  the Gauss map  $\mathcal{G}$ takes its values in the   hyperplane $z_{k+2} ={\lambda}_{k+2} z_{2}$ for some constant ${\lambda}_{k+2}  \in {\mathbb{C}}^{*}$. We see that  $\mathcal{G}\left(\Sigma\right)$  lies on a hyperplane $z_{k+2} = \left({\lambda}_{k+2} {\lambda}\right) z_{1}$.\\
   We conclude that, for each $k \in \{2, \cdots, n\}$, the Gauss map image $\mathcal{G}\left(\Sigma\right)$
  lies on some hyperplane  $z_{k} = {\alpha}_{k} z_{1}$. Hence, the Gauss map   $\mathcal{G}=[{z}_{1}: \cdots: {z}_{n+2} ]$ is constant. \\
 (b)  We consider the two cases. \\ (b1) If $\frac{\partial f_{k}}{\partial y}$  is a non-zero constant function, we can  take $(i,j)=(k+2,1)$ in (b) of Lemma \ref{entire} to obtain  $\frac{\partial f_{k}}{\partial x} + \lambda \frac{\partial f_{k}}{\partial y} =\frac{1}{{\lambda}_{k+2}}$ for some constant $\lambda \in  {\mathbb{C}}^{*}$. Hence, $\frac{\partial f_{k}}{\partial x}$  is also a constant function. \\ (b2) In the case when $\frac{\partial f_{k}}{\partial x}$  is a non-zero constant function, we  take $(i,j)=(k+2,2)$ in (b) of Lemma \ref{entire} to write  $\frac{\partial f_{k}}{\partial x} + \lambda \frac{\partial f_{k}}{\partial y} =\frac{\lambda}{{\lambda}_{k+2}}$  for some constant $\lambda \in  {\mathbb{C}}^{*}$.
 Thus, $\frac{\partial f_{k}}{\partial y}$  is  also  constant.
\end{proof}

\begin{corollary}  \label{ent3}
 Let $\Sigma$ be an entire minimal graph $\mathbf{\Phi}(x,y)=x {\mathbf{e}}_{1}+y{\mathbf{e}}_{2}+ f(x,y) {\mathbf{e}}_{3} +  g(x,y)  {\mathbf{e}}_{4}$ in ${\mathbb{R}}^{4}$. If $\frac{\partial f}{\partial x}$ is a non-zero constant function,
 then   $\Sigma$ is a plane.
\end{corollary}

\begin{proof}
By (b) of Theorem \ref{ent1}, the function $f(x,y)$ is affine. So, the surface $\Sigma$ can be viewed
as an entire minimal graph in some affine subspace ${\mathbb{R}}^{3}$. Bernstein's Theorem says that
$\Sigma$ is a plane. (Alternatively, since the function $f$ has bounded gradient,  Simon's Theorem \cite{HSHV09, Sim77} implies that the graph $\Sigma$ is a plane.)
\end{proof}

\begin{corollary}  \label{ent2}
Let  $\Sigma$ be an entire minimal graph of  $f=\left(f_{1}, \cdots, f_{n} \right):{\mathbb{R}}^{2} \rightarrow {\mathbb{R}}^{n}$, $n \geq 2$. If the mappings $(x,y) \mapsto \left(f_{i}(x,y), f_{j}(x,y)\right)$ are diffeomorphisms from ${\mathbb{R}}^{2}$ to itself  for all pairs $(i,j)$ with $1 \leq i < j \leq n$, then $\Sigma$  is a plane.
\end{corollary}

\begin{proof}
By (a) in Lemma \ref{oentire}, we can take a linear diffeomorphism from ${\mathbb{R}}^{2}$ to itself
\[
\left(u, v \right) \mapsto \left(x,y\right)=\left(u,au+bv\right)
\]
such that $u+iv$ becomes a global isothermal parameter for the graph $\Sigma$.  It follows that the composition map
\[
(u,v) \mapsto (x,y) \mapsto \left(f_{i}(x,y), f_{j}(x,y) \right)
\]
is the harmonic diffeomorphism from ${\mathbb{R}}^{2}$ to itself. Thus, it is linear. Since the first map $(u,v) \mapsto (x,y)$ is a linear diffeomorphism, the second map $(x,y) \mapsto \left(f_{i}(x,y), f_{j}(x,y) \right)$ is also linear.
\end{proof}

\begin{remark}   \label{OssUse}
Schoen \cite{Sch93} and Ni \cite{Ni02} proved the case $n=2$ in Corollary \ref{ent2}.
 Brendle and Warren \cite{BW10} proved that, given  two uniformly convex domains in ${\mathbb{R}}^{n}$ with smooth boundary, there exists a diffeomorphism  from the first domain to the second one such that its graph becomes a minimal Lagrangian
 submanifold of ${\mathbb{R}}^{n} \times {\mathbb{R}}^{n}$.
  \end{remark}

\subsection{Construction of special Lagrangian graphs in ${\mathbb{C}}^{2}$} \label{SR2R4}

  The Harvey-Lawson Theorem \cite{HL82} indicates that the special Lagrangian equation in ${\mathbb{C}}^{2}$ is the first integral of the minimal surface system for the gradient graph in ${\mathbb{R}}^{4}$.

\begin{theorem}[\textbf{From minimal graphs in ${\mathbb{R}}^{n+2}$ to special Lagrangian
graphs in ${\mathbb{C}}^{2}={\mathbb{R}}^{2} \times i {\mathbb{R}}^{2}$}]  \label{RnR4}
  Let $x{\mathbf{e}}_{1}+y{\mathbf{e}}_{2}+f_{1}(x,y){\mathbf{e}}_{3} +\cdots+f_{n}(x,y){\mathbf{e}}_{n+2}$  be a minimal
  graph  in ${\mathbb{R}}^{n+2}$ defined on the simply connected domain $\Omega \subset {\mathbb{R}}^{2}$ with the
  induced metric  $ds^{2}=E dx^{2}+2F dx dy + Gdy^{2}$ with $\omega=\sqrt{EG-F^{2}}$. Then, in ${\mathbb{R}}^{2} \times {\mathbb{R}}^{2}$  equipped with the standard symplectic structure, there exists a minimal
  gradient graph  $x{\mathbf{e}}_{1}+y{\mathbf{e}}_{2}+M(x,y){\mathbf{e}}_{3} +N(x,y){\mathbf{e}}_{4}$ satisfying the equalities
  \[
   M_{x} = \frac{E}{\omega}, M_{y} = \frac{F}{\omega},\;  N_{x} = \frac{F}{\omega}, \; N_{y} = \frac{G}{\omega}.
  \]
\end{theorem}

\begin{proof}
As in the proof of Lemma \ref{entire}, we begin with the two identities:
 \[
 \frac{\partial }{\partial x} \left(  \frac{F}{\omega}   \right) =  \frac{\partial }{\partial y} \left( \frac{E}{\omega}   \right) \quad \text{and} \quad
 \frac{\partial }{\partial x} \left( \frac{G}{\omega} \right) = \frac{\partial }{\partial y} \left( \frac{F}{\omega} \right).
 \]
 Since $\Omega$ is simply connected, Poincar\'{e} Lemma guarantees the existence of functions
 $N, M:\Omega \rightarrow \mathbb{R}$
 satisfying the equalities
  \[
 M_{x} = \frac{E}{\omega}, M_{y} = \frac{F}{\omega}, N_{x} = \frac{F}{\omega}, N_{y} = \frac{G}{\omega}.
 \]
 The  map $(x,y) \mapsto \left(M(x,y), N(x,y) \right)$ is area-preserving because of the equality
 \[
  \frac{\partial (M,N)}{\partial (x,y)} = \frac{EG-F^{2}}{{\omega}^{2}}=1.
 \]
 The equality  $M_{y} = \frac{F}{\omega} =  N_{x}$ shows that, by Poincar\'{e} Lemma again,
 $(x,y) \mapsto \left(M(x,y), N(x,y) \right)$ is a gradient map $\left(M, N \right)=\left( h_{x}, h_{y} \right)$
 for some function $h:\Omega \rightarrow \mathbb{R}$. The function $h$ satisfies the unimodular Hessian equation
\[
h_{xx}h_{yy}-{h_{xy}}^{2}=\frac{\partial (M,N)}{\partial (x,y)} =1.
\]
 We then observe that the function $h$ solves the special Lagrangian equation
  \[
      \cos \theta \left( {h}_{xx} +  {h}_{yy} \right)
     +  \sin \theta ( 1 -   {h}_{xx}  {h}_{yy} +  {{h}_{xy}}^{2}  )=0
  \]
 with $\theta=\frac{\pi}{2}$. The Harvey-Lawson Theorem \cite{HL82} implies that the gradient graph
 \[
 x{\mathbf{e}}_{1}+y{\mathbf{e}}_{2}+M(x,y){\mathbf{e}}_{3} +N(x,y){\mathbf{e}}_{4}
 =x{\mathbf{e}}_{1}+y{\mathbf{e}}_{2}+h_{x}{\mathbf{e}}_{3} +h_{y}{\mathbf{e}}_{4}
 \]
 becomes a special Lagrangian surface.
\end{proof}

\begin{remark} In section \ref{SyMA}, we show that this correspondence from  minimal graphs to special Lagrangian graphs
induces the symplectic graph rotation in Lemma \ref{Hone} which says that a solution of the symplectic Monge-Amp\`{e}re equations in two variables induces a solution of the unimodular Hessian equation in two variables.
\end{remark}

\begin{example}[\textbf{From catenoids to Lagrangian catenoids}] \label{GS1}
We consider the catenoid $z=\rho \, arcosh \left(  \frac{  \sqrt{x^{2} +y^{2} } }{ \rho } \right)$ in ${\mathbb{R}}^{3}$, where $\rho>0$ is a constant. Applying Theorem \ref{RnR4} to the catenoid, we find that the gradient map
 \[
   (x,y) \mapsto \left(  \sqrt{  1 - \frac{{\rho}^{2} }{  x^2 +y^2 } } \, x,   \sqrt{  1 - \frac{{\rho}^{2} }{  x^2 +y^2 } } \, y \right)
 \]
 is area-preserving. Its graph known as a Lagrangian catenoid. Castro and Urbano
 \cite{CaUr99} presented several geometric characterizations of the Lagrangian catenoids \cite[Theorem 3.5]{HL82}
 and \cite[Example 2]{Oss78}.
\end{example}

\begin{example}[\textbf{From helicoids to Lagrangian catenoids}] \label{GS2}
Let $\rho>0$ be a constant. The application of Theorem \ref{RnR4} to the helicoid $z=\rho \arctan \left(  \frac{ y }{ x } \right)$ in ${\mathbb{R}}^{3}$  gives the area-preserving  gradient map
 \[
   (x,y) \mapsto \left(  \sqrt{  1 + \frac{{\rho}^{2} }{  x^2 +y^2 } } \, x,   \sqrt{  1 + \frac{{\rho}^{2} }{  x^2 +y^2 } } \, y \right).
 \]
Its graph is also a Lagrangian catenoid. The area-preserving map in Example \ref{GS1} is the inverse of
 the area-preserving map in Example \ref{GS2}.
\end{example}

 \begin{example}[\textbf{From Scherk surfaces to Lagrangian Scherk surfaces}]
Let $\rho>0$ be a constant. Under the correspondence in Theorem \ref{RnR4}, Scherk's graph
\[
z=\frac{1}{\rho} \left[ \ln \left( \cos \left( \rho x \right) \right) -  \ln \left( \cos \left( \rho y \right) \right)     \right]
\]
 yields a special Lagrangian graph
 \[
   x{\mathbf{e}}_{1}+y{\mathbf{e}}_{2}+  \frac{1}{\rho} arcsinh \left[ \tan \left( \rho x \right)   {\cos} \left( \rho y \right) \right]  {\mathbf{e}}_{3} +  \frac{1}{\rho} arcsinh \left[ \tan \left( \rho y \right)   {\cos} \left( \rho x \right) \right]
    {\mathbf{e}}_{4}.
 \]
 It might be interesting to find some geometric characterizations of this surface.
\end{example}

\section{Twin correspondence and its applications} \label{TW}

 \subsection{Existence of twin correspondence}   \label{SecTwin}

Our aim is to generalize Calabi's correspondence \cite{Cal70}  between
the minimal surface equation  in  ${\mathbb{R}}^{3}$ and the maximal surface equation  in  ${\mathbb{L}}^{3}$
to higher codimension $n \geq 2$. We recall that the notion of the positive area-angle map was introduced in
Section \ref{ntas}.  Our ambient spaces  are  the  Euclidean space ${\mathbb{R}}^{n+2}$ endowed with the
metric ${dx_{1}}^{2} +\cdots + {dx_{n+2}}^{2}$ and  the  pseudo-Euclidean space  ${\mathbb{R}}^{n+2}_{n}$
equipped with the metric ${dx_{1}}^{2} + {dx_{2}}^{2} -  {dx_{3}}^{2} -\cdots - {dx_{n+2}}^{2}$.

\begin{theorem}[\textbf{Twin correspondence - version A}] \label{twinBABY}
There exists the twin correspondence (up to vertical translations) between two dimensional minimal graphs in  ${\mathbb{R}}^{n+2}$ carrying a positive area-angle function and two dimensional maximal graphs  in  ${\mathbb{R}}^{n+2}_{n}$ carrying the same positive area-angle function.
 More explicitly, we have the twin correspondence between the minimal graph    in ${\mathbb{R}}^{n+2}$
  \[
 x{\mathbf{e}}_{1}+y{\mathbf{e}}_{2}+f_{1}(x,y){\mathbf{e}}_{3} +\cdots+f_{n}(x,y){\mathbf{e}}_{n+2}
 \]
 over the simply connected domain $\Omega$ satisfying the positive area-angle condition
\[
\sum_{1 \leq i<j \leq n}  {\frac{ \partial \left( {f}_{i}, {f}_{j} \right)}{\partial (x, y)}}^{2}  <1
\]
and the maximal graph in ${\mathbb{R}}^{n+2}$
\[
  x{\mathbf{e}}_{1}+y{\mathbf{e}}_{2}+g_{1}(x,y){\mathbf{e}}_{3} +\cdots+g_{n}(x,y){\mathbf{e}}_{n+2}
\]
 over the same domain $\Omega$ satisfying the positive area-angle condition
\[
\sum_{1 \leq i<j \leq n}  {\frac{ \partial \left( {g}_{i}, {g}_{j} \right)}{\partial (x, y)}}^{2}  <1.
\]
\end{theorem}

For the construction of the twin correspondence in Theorem \ref{twinBABY}, we need to restate it even more explicitly:

 \begin{theorem}[\textbf{Twin correspondence - version B}] \label{twin1}
 (a) If the graph $\mathbf{\Phi}$ of  a positive area-angle map  $f=\left({f}_{1}, \cdots, {f}_{n}\right): \Omega \rightarrow {\mathbb{R}}^{n}$ with the area-angle ${\Theta} \in \left(0, \frac{\pi}{2} \right]$ becomes
a minimal surface in   ${\mathbb{R}}^{n+2}$, then
there exists  a positive area-angle map $g=\left({g}_{1}, \cdots, {g}_{n}\right): \Omega \rightarrow {\mathbb{R}}^{n}$ with
the same area-angle ${\Theta} \in \left(0, \frac{\pi}{2} \right]$ such that its graph $\mathbf{\widehat{\Phi}}$ becomes
a maximal surface in  ${\mathbb{R}}^{n+2}_{n}$.  \\
 (b) Conversely, if the graph $\mathbf{\widehat{\Phi}}$ of  a positive area-angle map $g=\left({g}_{1}, \cdots, {g}_{n}\right): \Omega \rightarrow {\mathbb{R}}^{n}$ with the area-angle ${\Theta} \in \left(0, \frac{\pi}{2} \right]$ is a maximal surface in  ${\mathbb{R}}^{n+2}_{n}$, then we can associate  a positive area-angle map  $f=\left({f}_{1}, \cdots, {f}_{n}\right): \Omega \rightarrow {\mathbb{R}}^{n}$ with the same area-angle ${\Theta} \in \left(0, \frac{\pi}{2} \right]$ satisfying that
 its graph $\mathbf{\Phi}$ becomes a minimal surface in ${\mathbb{R}}^{n+2}$.\\
 (c) The twin correspondence $\mathbf{\Phi} \Leftrightarrow \mathbf{\widehat{\Phi}}$ in (a) and (b) fulfils the following relations:\\
 (c1) The twin relations hold:
\[
    \left( \frac{\partial g_{k}}{\partial x}, \frac{\partial g_{k}}{\partial y} \right)
    = \left(  -  \frac{E}{\omega} \frac{\partial f_{k}}{\partial y} + \frac{F}{\omega} \frac{\partial f_{k}}{\partial x},
     \frac{G}{\omega}\frac{\partial f_{k}}{\partial x} - \frac{F}{\omega} \frac{\partial g_{k}}{\partial y} \right),  \quad k \in \{1, \cdots, n\},
\]
or equivalently,
\[
 \left( \frac{\partial f_{k}}{\partial x}, \frac{\partial f_{k}}{\partial y} \right)
    = \left(  \frac{\widehat{E}}{\widehat{\omega}} \frac{\partial g_{k}}{\partial y} - \frac{\widehat{F}}{\widehat{\omega}} \frac{\partial g_{k}}{\partial x},
   -  \frac{\widehat{G}}{\widehat{\omega}}\frac{\partial g_{k}}{\partial x} + \frac{\widehat{F}}{\widehat{\omega}} \frac{\partial g_{k}}{\partial y} \right), \quad k \in \{1, \cdots, n\}.
\]
Here, $ds_{\mathbf{\Phi}}^{2}=E dx^{2}+2F dx dy + Gdy^{2}$ denotes the induced metric and $\omega=\sqrt{EG-F^{2}}$. Likewise, we write $ds_{\mathbf{\widehat{\Phi}}}^{2}= \widehat{E}dx^{2}+2 \widehat{F}dx dy + \widehat{G}dy^{2}$ and $\widehat{\omega}=\sqrt{\widehat{E}\widehat{G}-{\widehat{F}}^{2}}$. \\
(c2) They share the area-angle. In fact each Jacobian determinant is preserved:
\[
       \frac{ \partial \left( f_{i}, f_{j} \right)}{\partial  (x, y)} =  \frac{ \partial \left( g_{i}, g_{j} \right)}{\partial  (x, y)}, \quad i, j   \in \{1, \cdots, n\}.
\]
(c3) The angle duality holds
\[
 \widehat{\omega}   \omega  =   {\sin}^{2} \Theta >0.
\]
(c4) Two induced metrics are conformally equivalent. In fact,
$ds_{\mathbf{\Phi}}^{2}= \frac{\omega}{\widehat{\omega}}
ds_{\mathbf{\widehat{\Phi}}}^{2}$. \\
(d)  Since two integrability conditions in (c1) are equivalent, we see that
the twin correspondence is involutive. The twin minimal surface
 $\widehat{\widehat{\Sigma}}$ of the twin surface $\widehat{\Sigma}$ of a minimal surface ${\Sigma}$
 is congruent to  ${\Sigma}$ up to vertical translations.
\end{theorem}

  \begin{proof} \label{twinproof}
  We show (a) and (c) simultaneously.  Working backwards gives (b). \\
  Let   $(x,y) \in \Omega \mapsto \mathbf{\Phi}(x,y)=x {\mathbf{e}}_{1}+y{\mathbf{e}}_{2}+ f_{1}(x,y) {\mathbf{e}}_{3} +  \cdots + f_{n}(x,y)  {\mathbf{e}}_{n+2}$ be a minimal surface  in   ${\mathbb{R}}^{n+2}$.  Write $ds_{\mathbf{\Phi}}^{2}=E dx^{2}+2F dx dy + Gdy^{2}$ and $\omega=\sqrt{EG-F^{2}}>0$. We further assumed the positive area-angle condition
\[
  1 > \cos \Theta = \Vert \mathcal{J} \Vert :=  \sqrt{ \sum_{1 \leq i<j \leq n}  {{\mathcal{J}}_{i,j}}^{2}}, \quad   {\mathcal{J}}_{i,j}:= \frac{ \partial \left( f_{i}, f_{j} \right)}{\partial (x, y)}.
\]
After setting $\left({\alpha}_{k}, {\beta}_{k} \right)=\left(\frac{\partial f_{k}}{\partial x}, \frac{\partial f_{k}}{\partial y} \right)$, we obtain
\[
 E= 1 + \sum_{k=1}^{n} {{\alpha}_{k}}^{2}, \; F=  \sum_{k=1}^{n}  {\alpha}_{k} {\beta}_{k}, \;  G= 1+  \sum_{k=1}^{n} {{\beta}_{k}}^{2}, \; {\mathcal{J}}_{i,j}=  {\alpha}_{i} {\beta}_{j} -  {\alpha}_{j} {\beta}_{i}.
\]
Then, Lagrange's Identity gives
\[
 {\omega}^{2}= EG-F^{2} = 1 + \sum_{k=1}^{n} {{\alpha}_{k}}^{2} + \sum_{k=1}^{n} {{\beta}_{k}}^{2} + \sum_{1 \leq i < j \leq n}  {{\mathcal{J}}_{i,j}}^{2}
\]
and thus
\[
  {\Vert \mathcal{J} \Vert}^{2} =\sum_{1 \leq i < j \leq n}  {{\mathcal{J}}_{i,j}}^{2} =  {\omega}^{2} - 1 - (E-1) -(G-1)=  {\omega}^{2} + 1 - E -G.
\]
Since the minimal surface system reads
\[
  \frac{\partial }{\partial x} \left( \frac{G}{\omega} {\alpha}_{k} - \frac{F}{\omega}  {\beta}_{k}  \right)+
   \frac{\partial }{\partial y} \left( \frac{E}{\omega} {\beta}_{k} - \frac{F}{\omega}  {\alpha}_{k}   \right)=0, \quad k \in \{1, \cdots, n\},
\]
and since $\Omega$ is simply connected, Poincar\'{e} Lemma
guarantees the existence of functions $g_{1}$, $\cdots$, $g_{n}:
\Omega \rightarrow {\mathbb{R}}$ satisfying the following integrability condition
\[
    \left( \frac{\partial g_{k}}{\partial x}, \frac{\partial g_{k}}{\partial y} \right)
    = \left(  -  \frac{E}{\omega} {\beta}_{k}  + \frac{F}{\omega}  {\alpha}_{k},
     \frac{G}{\omega}  {\alpha}_{k} - \frac{F}{\omega} {\beta}_{k}  \right), \quad k \in \{1, \cdots, n\}.
\]
Our goal is to show that
\[
\mathbf{\widehat{\Phi}}(x,y)=x {\mathbf{e}}_{1}+y{\mathbf{e}}_{2}+ g_{1}(x,y) {\mathbf{e}}_{3} +  \cdots + g_{n}(x,y)  {\mathbf{e}}_{n+2}
\]
becomes a maximal surface in ${\mathbb{R}}^{n+2}_{n}$. Write $ds_{\mathbf{\widehat{\Phi}}}^{2}= \widehat{E}dx^{2}+2 \widehat{F}dx dy + \widehat{G}dy^{2}$. Then,
\[
\widehat{E}= 1 - \sum_{k=1}^{n} {{\widehat{\alpha}}_{k}}^{2}, \; \widehat{F}= - \sum_{k=1}^{n}  {\widehat{\alpha}}_{k} {\widehat{\beta}}_{k}, \; \widehat{G}= 1-\sum_{k=1}^{n} {{\widehat{\beta}}_{k}}^{2}, \quad \left({\widehat{\alpha}}_{k}, {\widehat{\beta}}_{k} \right)=\left(\frac{\partial g_{k}}{\partial x}, \frac{\partial g_{k}}{\partial y} \right).
\]
 The first part of (c1) is now finished because the above condition reads
 \[
     \left( {\widehat{\alpha}}_{k} , {\widehat{\beta}}_{k}  \right)
     = \left(  -  \frac{E}{\omega} {\beta}_{k}  + \frac{F}{\omega}  {\alpha}_{k},
      \frac{G}{\omega}  {\alpha}_{k} - \frac{F}{\omega} {\beta}_{k}  \right).
 \]
 \textbf{Claim A.} We show the spacelike condition $\widehat{E}\widehat{G}-{\widehat{F}}^{2}>0$. It requires several steps. \\
 \textbf{Step A1.} We compute Jacobian determinants. The above twin relation reads
\[
  \begin{pmatrix} {\widehat{\beta}}_{k}  \\ -{\widehat{\alpha}}_{k}   \end{pmatrix}
 = \begin{pmatrix}
    \frac{G}{\omega}   &   -\frac{F}{\omega}  \\
    -\frac{F}{\omega}  &    \frac{E}{\omega}
  \end{pmatrix}
 \begin{pmatrix}  {{\alpha}}_{k}  \\ {{\beta}}_{k}      \end{pmatrix}
 \quad \text{or} \quad
 \begin{pmatrix}  {{\alpha}}_{k}  \\ {{\beta}}_{k}   \end{pmatrix}     =    \begin{pmatrix}
    \frac{E}{\omega}   &   \frac{F}{\omega}  \\
    \frac{F}{\omega}  &    \frac{G}{\omega}
     \end{pmatrix}
 \begin{pmatrix} {\widehat{\beta}}_{k}  \\ -{\widehat{\alpha}}_{k}  \end{pmatrix}.
\]
Equating the determinant of the both sides in
\[
\begin{pmatrix}   {{\alpha}}_{i} & {{\alpha}}_{j}  \\  {{\beta}}_{i} & {{\beta}}_{j}    \end{pmatrix}  =   \begin{pmatrix}
    \frac{E}{\omega}   &   \frac{F}{\omega}  \\
    \frac{F}{\omega}  &    \frac{G}{\omega}
 \end{pmatrix}
  \begin{pmatrix}  {\widehat{\beta}}_{i} &  {\widehat{\beta}}_{j}  \\ -{\widehat{\alpha}}_{i} & -{\widehat{\alpha}}_{j}   \end{pmatrix}
\]
 yields the equality ${\alpha}_{i} {\beta}_{j} - {\alpha}_{j} {\beta}_{i}= {\widehat{\alpha}}_{i} {\widehat{\beta}}_{j} - {\widehat{\alpha}}_{j} {\widehat{\beta}}_{i}$, which gives the assertion (c2):
\[
     {\mathcal{J}}_{i,j} :=  \frac{ \partial \left( f_{i}, f_{j} \right)}{\partial  (x, y)}={\alpha}_{i} {\beta}_{j} - {\alpha}_{j} {\beta}_{i}= {\widehat{\alpha}}_{i} {\widehat{\beta}}_{j} - {\widehat{\alpha}}_{j} {\widehat{\beta}}_{i} =  \frac{ \partial \left( g_{i}, g_{j} \right)}{\partial  (x, y)}, \quad i, j   \in \{1, \cdots, n\}.
\]
 \textbf{Step A2.} Here, our aim is to prove the equality
\[
\left( {{\widehat{\alpha}}_{1}}^{2} + \cdots + {{\widehat{\alpha}}_{n}}^{2} \right)
+\left( {{\widehat{\beta}}_{1}}^{2} + \cdots + {{\widehat{\beta}}_{n}}^{2} \right)=
\frac{ (E+G+2) {\omega}^{2}  - (E+G)^{2}  }{{\omega}^{2}}.
\]
Using the twin relation and the definition ${\omega}^{2}= EG-F^{2}$, we obtain
\begin{eqnarray*}
  &&  \sum_{k=1}^{n} \left[ {{\widehat{\alpha}}_{k}}^{2} + {{\widehat{\beta}}_{k}}^{2} \right]  \\
   & = & \sum_{k=1}^{n} \left[ { \left( -  \frac{E}{\omega} {\beta}_{k}  + \frac{F}{\omega}  {\alpha}_{k} \right) }^{2} + { \left(  \frac{G}{\omega}  {\alpha}_{k} - \frac{F}{\omega} {\beta}_{k}  \right)    }^{2} \right]   \\
   & = &  \frac{1}{{\omega}^{2}} \left[
 \left( F^2 +G^2 \right)  \sum_{k=1}^{n} {{\alpha}_{k}}^{2}  + \left( F^2 +E^2 \right)  \sum_{k=1}^{n} { {\beta}_{k}}^{2}
   - 2 (E+G)F \sum_{k=1}^{n}  {\alpha}_{k}  {\beta}_{k} \right] \\
 & = &  \frac{1}{{\omega}^{2}} \left[ \left( F^2 +G^2 \right) (E-1)  + \left( F^2 +E^2 \right) (G-1)
   - 2 (E+G) F^{2} \right]  \\
 &=& \frac{ (E+G+2) \left(EG-F^2\right)  - (E+G)^{2}  }{{\omega}^{2}}.
 \end{eqnarray*}
  \textbf{Step A3.} We here deduce the identity
  \[
  \widehat{E}\widehat{G}-{\widehat{F}}^{2} = \frac{\left( 1 - {\Vert \mathcal{J}  \Vert}^{2} \right)^{2} }{ {\omega}^{2}}.
  \]
 Then, the desired estimation $\widehat{E}\widehat{G}-{\widehat{F}}^{2}>0$ immediately follows from  our assumption $1 > \Vert \mathcal{J} \Vert$. We recall that  the Jacobian determinant equality ${\mathcal{J}}_{i,j} =  {\widehat{\alpha}}_{i} {\widehat{\beta}}_{j} - {\widehat{\alpha}}_{j} {\widehat{\beta}}_{i}$ holds. Using Lagrange's Identity again, we deduce
 \begin{eqnarray*}
 \widehat{E}\widehat{G}-{\widehat{F}}^{2} &=& \left(  1 - \sum_{k=1}^{n} {{\widehat{\alpha}}_{k}}^{2}
 \right)\left(1-\sum_{k=1}^{n} {{\widehat{\beta}}_{k}}^{2} \right) - \left(- \sum_{k=1}^{n}  {\widehat{\alpha}}_{k} {\widehat{\beta}}_{k} \right)^{2} \\
  &=& 1- \sum_{k=1}^{n} \left[ {{\widehat{\alpha}}_{k}}^{2} + {{\widehat{\beta}}_{k}}^{2} \right] +
  \left(   \sum_{k=1}^{n} {{\widehat{\alpha}}_{k}}^{2}
 \right)\left( \sum_{k=1}^{n} {{\widehat{\beta}}_{k}}^{2} \right) - \left(  \sum_{k=1}^{n}  {\widehat{\alpha}}_{k} {\widehat{\beta}}_{k} \right)^{2} \\
   &=& 1- \sum_{k=1}^{n} \left[ {{\widehat{\alpha}}_{k}}^{2} + {{\widehat{\beta}}_{k}}^{2} \right] +
  \sum_{1 \leq i < j \leq n}  {{\mathcal{J}}_{i,j}}^{2}  \\
     &=& 1- \frac{ (E+G+2) {\omega}^{2}  - (E+G)^{2}  }{{\omega}^{2}} +  \left(  {\omega}^{2} + 1 - E -G \right)  \\
  &=&  \frac{ \left( E+G - {\omega}^{2} \right)^{2} }{ {\omega}^{2}} \\
  &=& \frac{ \left( 1 - {\Vert \mathcal{J}  \Vert}^{2} \right)^{2} }{ {\omega}^{2}}.
 \end{eqnarray*}
 Since $1 > \cos \Theta = \Vert \mathcal{J} \Vert$, we deduce $\widehat{E}\widehat{G}-{\widehat{F}}^{2}>0$.
 Now, set $\widehat{\omega}:=\sqrt{\widehat{E}\widehat{G}-{\widehat{F}}^{2}}>0$. Then, the above equalities give
 the angle duality in (c3):
 \[
  \widehat{\omega}=\sqrt{\widehat{E}\widehat{G}-{\widehat{F}}^{2}}=\frac{E+G-{\omega}^{2}}{\omega}=\frac{ 1 - {\Vert \mathcal{J}  \Vert}^{2}  }{\omega}=\frac{  {\sin}^{2} \Theta  }{\omega}>0.
 \]
 \textbf{Claim B.} We verify that the graph $(x,y) \in \Omega \mapsto \mathbf{\widehat{\Phi}}(x,y)=x {\mathbf{e}}_{1}+y{\mathbf{e}}_{2}+ g_{1}(x,y) {\mathbf{e}}_{3} +  \cdots + g_{n}(x,y)  {\mathbf{e}}_{n+2}$ becomes a maximal surface in  pseudo-Euclidean space   ${\mathbb{R}}^{n+2}_{n}$. \\
  \textbf{Step B1.} We check (c4), which implies that the metric $ds_{\mathbf{\widehat{\Phi}}}^{2}= \frac{\widehat{\omega}}{\omega} ds_{\mathbf{\Phi}}^{2}$ is positive definite.  We need to  prove the three identities $\frac{E}{\omega}=\frac{\widehat{E}}{\widehat{\omega}}$,
$\frac{F}{\omega}=\frac{\widehat{F}}{\widehat{\omega}}$ and $\frac{G}{\omega}=\frac{\widehat{G}}{\widehat{\omega}}$.
The twin relation and the definition ${\omega}^{2}= EG-F^{2}$ yield
\begin{eqnarray*}
   \sum_{k=1}^{n}   {{\widehat{\alpha}}_{k}}^{2}
   & = & \sum_{k=1}^{n} { \left( -  \frac{E}{\omega} {\beta}_{k}  + \frac{F}{\omega}  {\alpha}_{k} \right) }^{2} \\
   & = & \frac{E^{2}}{{\omega}^{2}} \sum_{k=1}^{n} {{\beta}_{k}}^{2}
   +  \frac{F^{2}}{{\omega}^{2}} \sum_{k=1}^{n} {{\alpha}_{k}}^{2}
   -2 \frac{EF}{{\omega}^{2}} \sum_{k=1}^{n}  {\alpha}_{k} {\beta}_{k} \\
   &=&  \frac{E^{2}  (G-1) +  F^{2} (E-1)  -2EF^{2}} { {\omega}^{2}} \\
     &=&  \frac{E  {\omega}^{2} - E^2 -F^2 } { {\omega}^{2}} \\
 \end{eqnarray*}
and so, from the angle duality $ \widehat{\omega}=\frac{ 1 - {\Vert \mathcal{J}  \Vert}^{2}  }{\omega}=\frac{E+G-{\omega}^{2}}{\omega}$,
\[
\widehat{E} - \frac{\widehat{\omega}}{\omega} {E} =  1- \sum_{k=1}^{n}   {{\widehat{\alpha}}_{k}}^{2}  -
 \frac{\widehat{\omega}}{\omega}  E =  1- \frac{E  {\omega}^{2} - E^2 -F^2 } { {\omega}^{2}}   -
 \frac{E+G-{\omega}^{2}}{{\omega}^{2}}  E=0.
\]
The remaining two identities in (c4) can be proved similarly.  Thus, (c4) is proved. \\
 \textbf{Step B2.} We show that  the height function $g$ satisfies the maximal surface system. We
use the three identities in Step B1 to rewrite the twin relation
as
\[
\begin{pmatrix}  \frac{\partial f_{k}}{\partial x}   \\   \frac{\partial f_{k}}{\partial y} \end{pmatrix}
=\begin{pmatrix}  {{\alpha}}_{k}  \\ {{\beta}}_{k}   \end{pmatrix}     =    \begin{pmatrix}
    \frac{E}{\omega}   &   \frac{F}{\omega}  \\
    \frac{F}{\omega}  &    \frac{G}{\omega}
     \end{pmatrix}
 \begin{pmatrix} {\widehat{\beta}}_{k}  \\ -{\widehat{\alpha}}_{k}  \end{pmatrix}
 =   \begin{pmatrix}
\frac{\widehat{E}}{\widehat{\omega}}  &   \frac{\widehat{F}}{\widehat{\omega}}  \\
\frac{\widehat{F}}{\widehat{\omega}}  &   \frac{\widehat{G}}{\widehat{\omega}}
     \end{pmatrix}
 \begin{pmatrix} {\widehat{\beta}}_{k}  \\ -{\widehat{\alpha}}_{k}  \end{pmatrix}
 =   \begin{pmatrix}
\frac{\widehat{E}}{\widehat{\omega}}  {\widehat{\beta}}_{k}  - \frac{\widehat{F}}{\widehat{\omega}}  {\widehat{\alpha}}_{k}   \\
-\frac{\widehat{G}}{\widehat{\omega}}  {\widehat{\alpha}}_{k}  +  \frac{\widehat{F}}{\widehat{\omega}}  {\widehat{\beta}}_{k}
     \end{pmatrix},
\]
which is the second part of (c1). It therefore follows that, for all $k \in \{1, \cdots, n\}$,
\[
  \frac{\partial }{\partial x} \left( \frac{\widehat{G}}{\widehat{\omega}} {\widehat{\alpha}}_{k} - \frac{\widehat{F}}{\widehat{\omega}}  {\widehat{\beta}}_{k}  \right)+
   \frac{\partial }{\partial y} \left( \frac{\widehat{E}}{\widehat{\omega}} {\widehat{\beta}}_{k} - \frac{\widehat{F}}{\widehat{\omega}}  {\widehat{\alpha}}_{k}   \right)=   \frac{\partial }{\partial x} \left(-\frac{\partial f_{k}}{\partial y} \right) +  \frac{\partial }{\partial y} \left( \frac{\partial f_{k}}{\partial x} \right)= 0,
\]
which is the maximal surface system. This completes the proof of
(a).
\end{proof}

Bernstein's Theorem for two dimensional entire minimal graphs in Euclidean space with a positive area-angle holds \cite{HSHV09}.
As a consequence of this and the twin correspondence, we obtain

 \begin{corollary}  \label{LB}
  The only $2$-dimensional entire maximal graphs in ${\mathbb{R}}^{n+2}_{n}$, $n \geq 2$ with a positive area-angle
  function are spacelike planes.
 \end{corollary}

\begin{proof}
  Let $\widehat{\Sigma}$  be a $2$-dimensional entire maximal graph in ${\mathbb{R}}^{n+2}_{n}$ with a positive area-angle.
  Hence its Jacobian is bounded. Taking $\Omega={\mathbb{R}}^{2}$ in Theorem \ref{twin1}, we find that
 its twin minimal graph  $\Sigma$ in  ${\mathbb{R}}^{n+2}$ is also entire. By Theorem of Hasanis, Savas-Halilaj and Vlachos \cite{HSHV09}, we find that the entire minimal graph $\Sigma$ with bounded Jacobian is a plane. Since the twin correspondence transforms the planes in  ${\mathbb{R}}^{n+2}$  to the spacelike planes in ${\mathbb{R}}^{n+2}_{n}$, we conclude that $\widehat{\Sigma}$ becomes a spacelike plane.
\end{proof}

 \begin{corollary}[\textbf{Calabi's Theorem}, \cite{Cal70}]
  The only   entire maximal graphs in Lorentz space ${\mathbb{R}}^{3}_{1}$ are spacelike planes.
 \end{corollary}

\begin{proof}
   An  entire maximal graph $x {\mathbf{e}}_{1}+y{\mathbf{e}}_{2}+g(x,y) {\mathbf{e}}_{3}$ in  ${\mathbb{R}}^{3}_{1}$
  can be viewed as an  entire maximal graph
 $x {\mathbf{e}}_{1}+y{\mathbf{e}}_{2}+g(x,y) {\mathbf{e}}_{3}+0  {\mathbf{e}}_{4}$ in  ${\mathbb{R}}^{4}_{2}$
 with area-angle $\Theta=\frac{\pi}{2}$.
 \end{proof}

  Let $D^{2}$ denote the Hessian matrix operator. The twin correspondence induces the
following duality:

\begin{corollary}[\textbf{Duality between minimal gradient graphs in ${\mathbb{R}}^{4}$ and maximal gradient graphs
 in ${\mathbb{R}}^{4}_{2}$}] \label{SLE1} Let $\Omega \subset {\mathbb{R}}^{2}$ be a simply connected domain.
 If the gradient graph $x {\mathbf{e}}_{1}+y{\mathbf{e}}_{2}+{\widehat{p}}_{x} {\mathbf{e}}_{3} + {\widehat{p}}_{y} {\mathbf{e}}_{4}$ of the function $\widehat{p}: \Omega \rightarrow {\mathbb{R}}$
such that $\vert \det D^{2}\widehat{p} \,\vert<1$ is a maximal surface in  ${\mathbb{R}}^{4}_{2}$, then
the height function of the twin minimal graph in  ${\mathbb{R}}^{4}$ of the maximal graph $x {\mathbf{e}}_{1}+y{\mathbf{e}}_{2} -{\widehat{p}}_{y} {\mathbf{e}}_{3} + {\widehat{p}}_{x} {\mathbf{e}}_{4}$ becomes also a gradient map of some function  $p: \Omega \rightarrow {\mathbb{R}}$ with $\vert \det D^{2} {p} \, \vert<1$.
 \end{corollary}

\begin{remark}
We omit the proof of Corollary \ref{SLE1} because next section contains more general explanation.
Indeed, Lemma \ref{Hone} in Section \ref{SyMA} indicates that the special Lagrangian equation is
 equivalent to the split Special Lagrangian equation.
\end{remark}

\subsection{Symplectic Monge-Amp\'{e}re equations}  \label{SyMA}
We are interested in two symplectic Monge-Amp\'{e}re equations.
The first one is the well-known equation of special Lagrangian
graphs in the complex space ${\mathbb{C}}^{2}$
\[
 \cos \alpha \left( {F}_{xx} +  {F}_{yy} \right)  + \sin \alpha \left( 1 -  {F}_{xx}  {F}_{yy} + {{F}_{xy}}^{2}  \right)=0,
\]
and the second one is the equation of split Special Lagrangian  graphs
in the para-complex space ${\mathbb{D}}^{2}$
\[
  \cosh \beta \left( {F}_{xx} +  {F}_{yy} \right) + \sinh \beta \left( 1 +   {F}_{xx}  {F}_{yy} -  {{F}_{xy}}^{2}  \right)=0.
\]

\begin{remark}
 The even dimensional real vector space ${\mathbb{R}}^{2n}$ admits two different special Lagrangian geometries: \\
 (a)  As in \cite{Anc10a, Anc10b, Don09}, we may view ${\mathbb{R}}^{2n}$ as the complex space ${\mathbb{C}}^{n}$ endowed
 with the complex structure $J$ satisfying $J^{2}=-Id$ and pseudo-Hermitian form
 \[
    - \sum_{j=1}^{p} d z_{j} d{\overline{z}}_{j} +  \sum_{j=p+1}^{n} d z_{j} d{\overline{z}}_{j}
 \]
 where $0 \leq p \leq n$. The signature of their special Lagrangian submanifolds is $(p, n-p)$, so in particular,
 for $p \neq 0, n$, they have indefinite induced metric. Recently, Dong \cite{Don09} proved the instability of
 such special Lagrangian submanifolds and Anciaux \cite{Anc10b} showed that they minimize the volume in their
 Lagrangian homology classes. \\
 (b) On the other hand, adopting the split K\"{a}hler structure \cite{HL10, Mea91}, we here regard ${\mathbb{R}}^{2n}$
 as the para-complex space ${\mathbb{D}}^{n}$ endowed with the para-complex structure $I$ satisfying $I^{2}=Id$.
 Mealy \cite{Mea91} proved that the split special Lagrangian submanifolds (or equivalently, spacelike Lagrangian
 submanifolds of zero mean curvature) are homology maximizing. See also \cite[Theorem 5.3]{HL10}.
\end{remark}

 Recently, Kim, McCann and Warren \cite{KMW} studied an interesting relationship between the split
special Lagrangian geometry and classical mass transport problem.  The readers should consult \cite{HL10, Mea91} for more detailed background on the split special Lagrangian geometry.

\begin{lemma}[\textbf{Split special Lagrangian equation}] \label{splitSE} In pseudo-Euclidean
space ${\mathbb{R}}^{4}_{2}$, we consider a spacelike graph $\Sigma$ given by $(x,y,z,w)=\left(x,y,h_{x}(x,y),h_{y}(x,y)\right)$ 
for some ${\mathcal{C}}^{3}$ function $h: \Omega \rightarrow \mathbb{R}$ defined on a connected $xy$-domain $\Omega$. The following two statements are equivalent.\\
(a) The graph $\Sigma$ has zero mean curvature in  ${\mathbb{R}}^{4}_{2}$.\\
(b) The potential function $h$ satisfies the split  special Lagrangian equation
\[
      \cosh \theta \left( {h}_{xx} +  {h}_{yy} \right)
     +  \sinh \theta ( 1 +   {h}_{xx}  {h}_{yy} - {{h}_{xy}}^{2}  )=0
\]
for some constant angle $\theta \in \mathbb{R}$.
\end{lemma}

 \begin{proof} We give the details of the equivalence of (a) and  (b). Since the spacelike graph  $\Sigma$ in
   ${\mathbb{R}}^{4}_{2}$ admits the positive definite metric
\[
    ds^{2} = \left( 1 - {h_{xx}}^{2} - {h_{yx}}^2  \right) dx^{2}
    - 2 \left( h_{xx} h_{xy} + h_{yx} h_{yy} \right) dx dy +
     \left( 1 - {h_{xy}}^{2} - {h_{yy}}^2  \right) dy^{2},
\]
we require the inequality
\begin{eqnarray*}
    0&<&  \left(   1 - {h_{xx}}^{2} - {h_{yx}}^2    \right) \left(   1 - {h_{xy}}^{2} - {h_{yy}}^2  \right)
    -   \left( h_{xx} h_{xy} + h_{yx} h_{yy}\right)^{2} \\
   &=&  \left(   1 +   {h}_{xx}  {h}_{yy} - {{h}_{xy}}^{2}   \right)^{2} -  \left(
   h_{xx} + h_{yy} \right)^{2}.
\end{eqnarray*}
It guarantees that $1 +   {h}_{xx}  {h}_{yy} - {{h}_{xy}}^{2}$ never vanish on the domain $\Omega$. Furthermore, we
obtain the well-defined function $\phi: \Omega \rightarrow (-1,1)$ given by
\[
  \phi:= \frac{ {h}_{xx} +  {h}_{yy} }{ 1 +   {h}_{xx}  {h}_{yy} - {{h}_{xy}}^{2} }.
\]
The gradient graph $\left(x,y, f(x,y),g(x,y)\right)=\left(x,y,h_{x},h_{y}\right)$ has zero mean curvature in  ${\mathbb{R}}^{4}_{2}$ when its height function $(f, g)=\left(h_{x},h_{y}\right)$ solves the maximal surface system
\[
 \begin{cases}
   0= \left( 1 - {f_{y}}^{2} - {{g_{y}}^2} \right) f_{xx} + 2 \left( f_{x} f_{y} + g_{x} g_{y} \right) f_{xy}
    +    \left( 1 - {f_{x}}^{2} - {{g_{x}}^2} \right) f_{yy}, \\
   0= \left( 1 - {f_{y}}^{2} - {{g_{y}}^2} \right) g_{xx} + 2 \left( f_{x} f_{y} + g_{x} g_{y} \right) g_{xy}
    +    \left( 1 - {f_{x}}^{2} - {{g_{x}}^2} \right) g_{yy}, \\
  \end{cases}
\]
or equivalently,
\[
 \begin{cases}
   0= \left( 1 - {h_{xy}}^{2} - {{h_{yy}}^2} \right) h_{xxx} + 2 \left( h_{xx} + h_{yy} \right) h_{xy} h_{xxy}
    +    \left( 1 - {h_{xx}}^{2} - {{h_{xy}}^2} \right) h_{xyy}, \\
   0= \left( 1 - {h_{xy}}^{2} - {{h_{yy}}^2} \right) h_{yxx} + 2 \left( h_{xx} + h_{yy} \right) h_{xy} h_{yxy}
    +    \left( 1 - {h_{xx}}^{2} - {{h_{xy}}^2} \right) h_{yyy}.
 \end{cases}
\]
It can be re-written as
\[
 \begin{cases}
   0=  {\left( 1+   {h}_{xx}  {h}_{yy} -  {{h}_{xy}}^{2}  \right)} \frac{\partial }{\partial x}{\left( h_{xx} + h_{yy} \right)}
   - {\left( h_{xx} + h_{yy} \right)}  \frac{\partial }{\partial x}{\left( 1 + {h}_{xx}{h}_{yy} - {{h}_{xy}}^{2}  \right)}, \\
   0=  {\left( 1 +   {h}_{xx}  {h}_{yy} - {{h}_{xy}}^{2}  \right)} \frac{\partial }{\partial y}{\left( h_{xx} + h_{yy} \right)}
   - {\left( h_{xx} + h_{yy} \right)}  \frac{\partial }{\partial y}{\left( 1 + {h}_{xx}{h}_{yy} - {{h}_{xy}}^{2}  \right)}. \\
 \end{cases}
\]
It is equivalent to say that the gradient of $\phi$ vanishes on $\Omega$:
\[
0= {\phi}_{x}= \frac{\partial }{\partial x}{\left( \frac{ {h}_{xx} +  {h}_{yy} }{ 1 +   {h}_{xx}  {h}_{yy} - {{h}_{xy}}^{2} } \right)}
\;\; \text{and} \;\;
0= {\phi}_{y}= \frac{\partial }{\partial y}{\left( \frac{ {h}_{xx} +  {h}_{yy} }{ 1 +   {h}_{xx}  {h}_{yy} - {{h}_{xy}}^{2} } \right)}.
\]
Since its domain $\Omega$ is connected, this system holds only when the quotient function $\phi: \Omega \rightarrow (-1,1)$ is a constant $-\tanh \theta$ for some $\theta \in \mathbb{R}$.
\end{proof}

\begin{lemma}[\textbf{Symplectic graph rotations}] \label{Hone}
Let $F:\Omega \rightarrow \mathbb{R}$ be a function of class ${\mathcal{C}}^{2}$ satisfying
the symplectic Monge-Amp\'{e}re equation
\[
{\lambda}_{1} \left( F_{xx} + F_{yy} \right) + {\lambda}_{2} \left( 1
- \epsilon \left( F_{xx} F_{yy} -{F_{xy}}^{2} \right) \right)=0,
\]
where $\lambda_{1}, \lambda_{2} \in \mathbb{R}$ and $\epsilon \in \{-1, 1\}$ are constants satisfying that
${\lambda_{1}}^{2}+ \epsilon{\lambda_{2}}^{2}=1$.
Then, the induced function $h:\Omega \rightarrow \mathbb{R}$ defined by
\[
 h(x,y) = {\lambda}_{2} F(x,y) - \epsilon {\lambda}_{1} \frac{x^{2}+y^{2}}{2}
\]
satisfies the  unimodular Hessian equation $h_{xx} h_{yy} -{h_{xy}}^{2}=1$.
\end{lemma}

\begin{proof} We compute
\begin{eqnarray*}
   h_{xx} h_{yy} -{h_{xy}}^{2}
   &=& \left( {\lambda}_{2} F_{xx} - \epsilon {\lambda}_{1}   \right)\left(
    {\lambda}_{2} F_{yy} - \epsilon {\lambda}_{1}  \right)-{\left(  {\lambda}_{2} F_{xy}  \right)}^{2}  \\
   &=& {{\lambda}_{2}}^{2} \left( F_{xx} F_{yy} -{F_{xy}}^{2} \right) - \epsilon \lambda_{2} \lambda_{1}
   \left( F_{xx} + F_{yy} \right) +  {{\lambda}_{1}}^{2}   \\
   &=& {{\lambda}_{2}}^{2} \left( F_{xx} F_{yy} -{F_{xy}}^{2} \right) + \epsilon {\lambda_{2}}^{2}
   \left( 1 - \epsilon \left( F_{xx} F_{yy} -{F_{xy}}^{2} \right) \right) +  {{\lambda}_{1}}^{2}   \\
   &=& {{\lambda}_{1}}^{2} + \epsilon {{\lambda}_{2}}^{2} \\
   &=& 1.
\end{eqnarray*}
\end{proof}

\begin{remark}
By the correspondence in Theorem  \ref{RnR4}, the minimal graph
$\left(x,y,F_{x},F_{y}\right)$ in ${\mathbb{R}}^{4}$ corresponds to the minimal gradient graph
$\left(x,y,h_{x},h_{y}\right)$ in ${\mathbb{R}}^{4}$. This observation induces Lemma \ref{Hone}.
Under the correspondence in Lemma \ref{Hone}, any harmonic function $F(x,y)$ corresponds to the same function
$h(x,y) = \pm \frac{1}{2} \left( x^{2}+y^{2} \right)$.
\end{remark}

 \begin{theorem}[\textbf{Calabi type theorem for entire split special Lagrangian equation}] \label{CSSL}
 If a maximal surface in ${\mathbb{R}}^{4}_{2}$ is an entire  gradient  graph $(x,y,f_{x},f_{y})$ for some 
  $\mathcal{C}^{3}$ function $f:{\mathbb{R}}^{2} \rightarrow {\mathbb{R}}$, then it is a spacelike plane.
 \end{theorem}

 \begin{proof} First, Lemma \ref{splitSE} guarantees that, for some constant $\theta \in \mathbb{R}$, the potential
 function $f$ satisfies the split special Lagrangian equation
 \[
      \cosh \theta \left( {f}_{xx} +  {f}_{yy} \right)
     +  \sinh \theta ( 1 +  {f}_{xx}  {f}_{yy} -  {{f}_{xy}}^{2}  )=0.
  \]
 (a) When $\sinh \theta \neq 0$, since Lemma \ref{Hone} says that the entire function
 \[
  h(x,y) = \sinh \theta f(x,y) + \cosh \theta \frac{x^{2}+y^{2}}{2}
 \]
 satisfies the unimodular Hessian equation $h_{xx} h_{yy} -{h_{xy}}^{2}=1$, J\"{o}rgens'
 Theorem guarantees that $h$ is quadratic and that $f$ is also quadratic. \\
 (b) When  $\sinh \theta =0$, we see that the entire function $f$ is harmonic. In other words,
 the function $\phi: \mathbb{C} \rightarrow \mathbb{C}$, $z=x+iy \mapsto f_{x}-if_{y}$ is
 holomorphic. The induced metric of
 the maximal graph $(x,y,f_{x},f_{y})$ is conformal:
 \begin{eqnarray*}
   ds^{2} = \left( 1 - {f_{xx}}^{2} - {f_{yx}}^2  \right) \left( dx^{2} + dy^{2} \right).
 \end{eqnarray*}
 Since our graph $(x,y,f_{x},f_{y})$ is spacelike, we have $1 - {f_{xx}}^{2} - {f_{yx}}^2>0$ on the whole plane ${\mathbb{R}}^{2}$. Since the entire holomorphic function ${\phi}'(z)=f_{xx}-if_{yx}$ is bounded:
 \[
    \vert \phi'(z) \vert = \sqrt{ {f_{xx}}^{2} + {f_{yx}}^{2} } <1,
 \]
 Liouville's Theorem guarantees that three entire functions $f_{xx}$, $f_{yx}$ and $f_{yy}=-f_{xx}$ are
 constants. Hence, $f$ is quadratic.
 \end{proof}

 \begin{corollary}[\textbf{Bernstein type theorem for entire special Lagrangian equation}, Fu \cite{Fu98},
 Yuan \cite{YU02}] \label{FY} If a minimal surface in ${\mathbb{R}}^{4}$ becomes an entire gradient
 graph $(x,y,f_{x},f_{y})$ for some function $f: {\mathbb{R}}^{2} \rightarrow {\mathbb{R}}$ of class $\mathcal{C}^{3}$, then the potential function $f$ is harmonic or quadratic.
 \end{corollary}

 \begin{proof} The Harvey-Lawson Theorem shows that there exists a constant $\theta \in \mathbb{R}$ such that
  the potential function $f$ satisfies the special Lagrangian equation
 \[
      \cos \theta \left( {f}_{xx} +  {f}_{yy} \right)
     +  \sin \theta ( 1 -   {f}_{xx}  {f}_{yy} +  {{f}_{xy}}^{2}  )=0.
  \]
  When $\sin \theta \neq 0$, since the entire function
  $h(x,y) = \sin \theta f(x,y) - \cos \theta \frac{x^{2}+y^{2}}{2}$
 satisfies $h_{xx} h_{yy} -{h_{xy}}^{2}=1$, J\"{o}rgens'
 Theorem guarantees that $h$ and $f$ are quadratic.
 \end{proof}

\begin{remark} We provide two comments on the \textit{reverse} unimodular Hessian
equation $h_{xx}h_{yy}-{h_{xy}}^{2}=-1$. First, the graph
rotation, which is analogous to the construction in Lemma
\ref{Hone}, is available for the following two Monge-Amp\'{e}re
equations
\[
  \sinh \theta \left( F_{xx} + F_{yy} \right) + \cosh \theta \left( 1 + F_{xx}F_{yy} - {F_{xy}}^{2} \right)=0,
\]
and
\[
 \sin \theta \left( - F_{xx} + F_{yy} \right) + \cos \theta \left( 1 + F_{xx}F_{yy} - {F_{xy}}^{2} \right)=0.
\]
Indeed, whenever a function $F$ of class ${\mathcal{C}}^{2}$ satisfies
\[
{\lambda}_{1} \left( \epsilon \, F_{xx} + F_{yy} \right) + {\lambda}_{2} \left( 1 + F_{xx} F_{yy} -{F_{xy}}^{2} \right)=0,
\]
where $\lambda_{1}, \lambda_{2} \in \mathbb{R}$ and $\epsilon \in \{-1, 1\}$ are constants with
$-\epsilon {\lambda_{1}}^{2} + {\lambda_{2}}^{2}=1$, the new function $h(x,y) = {\lambda}_{2} F(x,y) + {\lambda}_{1} \frac{x^{2}+ \epsilon y^{2}}{2}$ satisfies $h_{xx} h_{yy} -{h_{xy}}^{2}=-1$. Second, we observe that
it admits entire solutions of the form $h(x,y) = xy + f(x)$.
 Chamberland \cite[Theorem 3.1]{Cha03} constructed a new entire solution of $h_{xx} h_{yy} -{h_{xy}}^{2}=-1$.
\end{remark}

 \subsection{Twin surfaces in simultaneous conformal coordinates} \label{SecHOL}
We are able to read the twin surfaces in simultaneous conformal coordinates and obtain the twin relation with
respect to their Weierstrass representation formulas.

 \begin{theorem}[\textbf{Reading twin surfaces in conformal coordinates}] \label{SCC}  Let $\mathbf{\Phi}(x,y)$ be the minimal graph  in   ${\mathbb{R}}^{n+2}$ of  a positive area-angle map $f: \Omega \rightarrow {\mathbb{R}}^{n}$ and $\mathbf{\widehat{\Phi}}(x,y)$ its twin maximal graph in ${\mathbb{R}}^{n+2}_{n}$  of
 a positive area-angle map $g: \Omega \rightarrow {\mathbb{R}}^{n}$. \\
  (a) There exists a simultaneous conformal coordinate $\xi={\xi}_{1}+i{\xi}_{2}$ for the minimal graph $\mathbf{\Phi}(x,y)$  in ${\mathbb{R}}^{n+2}$ and its twin maximal graph $\mathbf{\widehat{\Phi}}(x,y)$ in ${\mathbb{R}}^{n+2}_{n}$. \\
  (b) Letting  $\Psi: (x,y) \mapsto \left( {\xi}_{1}, {\xi}_{2} \right)$ be the coordinate transformation in (a),
  we see that the conformal harmonic immersion  $\mathbf{\Phi} \circ {\Psi}^{-1}$   induces  the holomorphic null curve
  \[
  2\frac{\partial}{\partial  {\xi}} \mathbf{\Phi} \circ {\Psi}^{-1}=\left({\phi}_{1}, {\phi}_{2}, {\phi}_{3}, \cdots, {\phi}_{n+2} \right)
  \]
  and  the conformal harmonic  immersion $\mathbf{\widehat{\Phi}} \circ {\Psi}^{-1}$  induces  the holomorphic
  null curve
  \[
  2\frac{\partial}{\partial  {\xi}} \mathbf{\widehat{\Phi}} \circ {\Psi}^{-1}= \left({\widehat{\phi}}_{1}, {\widehat{\phi}}_{2}, {\widehat{\phi}}_{3},   \cdots, {\widehat{\phi}}_{n+2} \right).
  \]
 Then, they obey the twin relation
 \[
 {\phi}_{1}={\widehat{\phi}}_{1}, \;\; {\phi}_{2}={\widehat{\phi}}_{2}, \;\; {\phi}_{k+2}=-i{\widehat{\phi}}_{k+2}, \;\; k \in \{1, \cdots, n\}.
 \]
  \end{theorem}

 \begin{proof} As in the proof of (a) in Lemma \ref{entire}, we take the coordinate transformation
\[
 \Psi: (x,y) \mapsto \left( {\xi}_{1}, {\xi}_{2} \right)
 \]
 such that
\[
 J_{\Psi}= \frac{ \partial \left(  {\xi}_{1}, {\xi}_{2} \right)}{\partial (x, y)} = \det
   \begin{pmatrix}    1+ \frac{E}{\omega}   &   \frac{F}{\omega}  \\     \frac{F}{\omega}  &  1+  \frac{G}{\omega}
   \end{pmatrix}    = 2+ \frac{E+G}{\omega}>2.
 \]
  Since $ J_{\Psi}=\frac{ \partial \left(  {\xi}_{1}, {\xi}_{2} \right)}{\partial (x, y)}>0$, we have the existence of
 the local inverse 
 \[
 \left( {\xi}_{1}, {\xi}_{2} \right) \mapsto (x,y).
 \]
Now,  using the Chain Rule, we obtain the conformal metrics:
\[
  ds_{\mathbf{\Phi}}^{2}=\frac{\omega}{J_{\Psi}} \left( d{{\xi}_{1}}^{2} +  d{{\xi}_{2}}^{2} \right) 
\]
and  
\[ ds_{\mathbf{\widehat{\Phi}}}^{2}=  \frac{\widehat{\omega}}{J_{\Psi}} \left( d{{\xi}_{1}}^{2} +  d{{\xi}_{2}}^{2} \right) =  \frac{\widehat{\omega}}{\omega} ds_{\mathbf{\Phi}}^{2}.
\]
Proof of (a) is finished. Next, we prove (b). We consider two conformal immersions
  \[
\mathbf{\Phi} \circ {\Psi}^{-1} \left(  {\xi}_{1}, {\xi}_{2} \right) =
x \left(  {\xi}_{1}, {\xi}_{2} \right) {\mathbf{e}}_{1}+y \left(  {\xi}_{1}, {\xi}_{2} \right){\mathbf{e}}_{2}+
\sum_{k=1}^{n} f_{k}(x \left(  {\xi}_{1}, {\xi}_{2} \right),y \left(  {\xi}_{1}, {\xi}_{2} \right)) {\mathbf{e}}_{k+2}
\]
and
\[
\mathbf{\widehat{\Phi}} \circ {\Psi}^{-1} \left(  {\xi}_{1}, {\xi}_{2} \right) =
x \left(  {\xi}_{1}, {\xi}_{2} \right) {\mathbf{e}}_{1}+y \left(  {\xi}_{1}, {\xi}_{2} \right){\mathbf{e}}_{2}+
\sum_{k=1}^{n} g_{k}(x \left(  {\xi}_{1}, {\xi}_{2} \right),y \left(  {\xi}_{1}, {\xi}_{2} \right)) {\mathbf{e}}_{k+2}.
\]
 By the definition of the induced holomorphic curves, obviously, we meet
 \[
 {\widehat{\phi}}_{1}={\phi}_{1} \;\, \text{and} \;\, {\widehat{\phi}}_{2}={\phi}_{2}.
 \]
  It now remains to check the equality
 \[
 {\phi}_{k+2}=-i{\widehat{\phi}}_{k+2}, \quad k \in \{1, \cdots, n\}.
 \]
 We need to show that
\[
\frac{\partial}{\partial \xi}  g_{k}(x \left(  {\xi}_{1}, {\xi}_{2} \right),y \left(  {\xi}_{1}, {\xi}_{2} \right))
=-i \left( \frac{\partial}{\partial \xi}  f_{k}(x \left(  {\xi}_{1}, {\xi}_{2} \right),y \left(  {\xi}_{1}, {\xi}_{2} \right)) \right)
\]
or  equivalently,
\[
   \left( \frac{\partial f_{k}}{\partial {\xi}_{1}}, \frac{\partial f_{k}}{\partial {\xi}_{2}} \right)
 = \left( \frac{\partial g_{k}}{\partial {\xi}_{2}}, - \frac{\partial g_{k}}{\partial {\xi}_{1}} \right).
\]
We only check the second components. The twin relation in Theorem \ref{twin1} reads
 \[
     \left( {\widehat{\alpha}}_{k} , {\widehat{\beta}}_{k}  \right)
     = \left(  -  \frac{E}{\omega} {\beta}_{k}  + \frac{F}{\omega}  {\alpha}_{k},
      \frac{G}{\omega}  {\alpha}_{k} - \frac{F}{\omega} {\beta}_{k}  \right),
\]
where
\[
  \left({\alpha}_{k}, {\beta}_{k}, {\widehat{\alpha}}_{k}, {\widehat{\beta}}_{k} \right)=\left(\frac{\partial f_{k}}{\partial x},   \frac{\partial f_{k}}{\partial y}, \frac{\partial g_{k}}{\partial x}, \frac{\partial g_{k}}{\partial y} \right).
  \]
  Also, we prepare the  equality
\[
\begin{pmatrix}   \frac{\partial x}{\partial {\xi}_{1}}    &   \frac{\partial x}{\partial {\xi}_{2}}   \\
     \frac{\partial y}{\partial {\xi}_{1}}    &   \frac{\partial y}{\partial {\xi}_{2}}       \end{pmatrix}  =
        {    \begin{pmatrix}   \frac{\partial {\xi}_{1}}{\partial x}    &   \frac{\partial {\xi}_{1}}{\partial y} \\
    \frac{\partial {\xi}_{2}}{\partial x}    &   \frac{\partial {\xi}_{2}}{\partial y}   \end{pmatrix} }^{-1}   =  \frac{1}{J_{\Psi}}
   \begin{pmatrix}    1+ \frac{G}{\omega}   &  - \frac{F}{\omega}  \\   -  \frac{F}{\omega}  &  1+  \frac{E}{\omega}
   \end{pmatrix}.
 \]
  We then use this together with the Chain Rule to deduce
\begin{eqnarray*}
  &&  \frac{\partial g_{k}}{\partial {\xi}_{1}} \\
  &=& \frac{\partial x}{\partial {\xi}_{1}}  \frac{\partial g_{k}}{\partial x} +  \frac{\partial y}{\partial {\xi}_{1}}  \frac{\partial g_{k}}{\partial y}  \\
  &=&\frac{\partial x}{\partial {\xi}_{1}} \left[  -  \frac{E}{\omega} {\beta}_{k}  + \frac{F}{\omega}  {\alpha}_{k} \right]  +  \frac{\partial y}{\partial {\xi}_{1}} \left[  \frac{G}{\omega}  {\alpha}_{k} - \frac{F}{\omega} {\beta}_{k}   \right]  \\
  &=& \left[ \frac{\partial x}{\partial {\xi}_{1}}  \frac{F}{\omega} + \frac{\partial y}{\partial {\xi}_{1}}
   \frac{G}{\omega}  \right]  {\alpha}_{k}   - \left[  \frac{\partial x}{\partial {\xi}_{1}}  \frac{E}{\omega} + \frac{\partial y}{\partial {\xi}_{1}}  \frac{F}{\omega}    \right]  {\beta}_{k}  \\
  &=&    \frac{1}{J_{\Psi}}  \left[    \left(  1+ \frac{G}{\omega}  \right)     \frac{F}{\omega} +
     \left( - \frac{F}{\omega}  \right)        \frac{G}{\omega}  \right]  {\alpha}_{k}   -
        \frac{1}{J_{\Psi}}  \left[     \left(  1+ \frac{G}{\omega}  \right)    \frac{E}{\omega} +     \left( - \frac{F}{\omega}  \right)  \frac{F}{\omega}    \right]  {\beta}_{k}  \\
  &=& \frac{1}{J_{\Psi}}  \frac{F}{\omega}  {\alpha}_{k} - \frac{1}{J_{\Psi}}  \left(  1+ \frac{E}{\omega}  \right)  {\beta}_{k}   \\
  &=& -  \frac{\partial x}{\partial {\xi}_{2}}   \frac{\partial f_{k}}{\partial x} -  \frac{\partial y}{\partial {\xi}_{2}}   \frac{\partial f_{k}}{\partial y}\\
   &=& -  \frac{\partial f_{k}}{\partial {\xi}_{2}}.
\end{eqnarray*}
\end{proof}

\begin{remark}
In the case when $n=1$, the twin correspondence in conformal coordinates, described in (b) of Theorem \ref{SCC},
is also observed in \cite{LLS00}. Ara\'{u}jo and Leite \cite{AL09} illustrated several interesting results
on Calabi's correspondence.
\end{remark}

\end{document}